\documentclass[10pt]{article}
\usepackage{graphics}
\usepackage{graphicx}

\usepackage[a4paper, left=35mm,right=35mm,top=34mm,bottom=34mm]{geometry}
\usepackage[utf8]{inputenc}
\usepackage[T1]{fontenc}
\usepackage[english]{babel}

\usepackage{enumerate}
\usepackage{graphicx}
\usepackage{hyperref}
\hypersetup{
    colorlinks=true,
    linkcolor=blue,
    filecolor=magenta,      
    urlcolor=cyan,
}
\usepackage{lipsum}
\usepackage{amsfonts}
\usepackage{graphicx}
\usepackage{epstopdf}
\usepackage{algorithmic}
\usepackage{lipsum}
\usepackage{amsfonts}
\usepackage{graphicx}
\usepackage{epstopdf}

\usepackage{hyperref}
\ifpdf
  \DeclareGraphicsExtensions{.eps,.pdf,.png,.jpg}
\else
  \DeclareGraphicsExtensions{.eps}
\fi

\usepackage{xcolor}

\usepackage{tikz}

\usepackage{marginnote,cite,amsmath,bbm}

 \usepackage{dsfont}
 \usepackage{psfrag}
 \usepackage{color}
 \usepackage{tikz}
 \usepackage{subfigure}
\usepackage{mathtools,amsthm,amssymb,amsfonts}
\usepackage{algorithm}
\usepackage{algorithmic}
\usepackage{tikz}
\usepackage{pgfplots}
\usepackage{siunitx}
\usepackage{psfrag}
\usepackage{color}
\usepackage{amssymb}
\usepackage{csquotes}
\usepackage{booktabs}
\usepackage[rightcaption]{sidecap}
\usepackage[colorinlistoftodos]{todonotes}
\usepackage{float}
\usepackage{libertine}
\usepackage{xurl}
\usepackage{csquotes}
\usepackage{graphics}
\usepackage{graphicx}
\def\R{\mathbb{R}}

\def\N{\mathbb{N}}
\def\C{\mathbb{C}}
\graphicspath{{./convergence exp}{./convergence resonnance}{./optimization}{./comparaison methode}{./comparaison loop projection}{./comparaison gen}}

\def\del  {\partial}
\def\eps{\varepsilon}

\def\R{\mathbb{R}}

\def\N{\mathbb{N}}

 \def\dx{{\rm d}x}

\def\Tr{\operatorname{Tr}}

\newtheorem{remark}{\textbf{Remark}}

\makeatother
\theoremstyle{plain}
\newtheorem{theorem}{\textbf{Theorem}}
\newtheorem{lemma}{\textbf{Lemma}}
\newtheorem{definition}{\textbf{Definition}}

\usepackage{caption} 
\usepackage{fancyhdr}

\author{
  {\normalsize Frédéric Magoulès, Mathieu Menoux, Anna Rozanova-Pierrat}\thanks{CentraleSup\'elec, Universit\'e Paris-Saclay, France.}
    		}
\title{Frequency range non-Lipschitz parametric optimization of a noise absorption}
\date{}

\begin{document}
\maketitle
\thispagestyle{fancy}

\begin{abstract}
\noindent In the framework of the optimal wave energy absorption, we solve theoretically and numerically a parametric shape optimization problem to find the optimal distribution of absorbing material in the reflexive one defined by a characteristic function in the Robin-type boundary condition associated with the Helmholtz equation. Robin boundary condition can be given on a part or the all boundary of a bounded $(\eps,\infty)$-domain of $\R^n$. The geometry of the partially absorbing boundary is fixed, but allowed to be non-Lipschitz, for example, fractal. It is defined as the support of a $d$-upper regular measure with $d\in ]n-2,n[$. Using the well-posedness properties of the model, for any fixed volume fraction of the absorbing material, we establish the existence of at least one optimal distribution minimizing the acoustical energy on a fixed frequency range of the relaxation problem. Thanks to the shape derivative of the energy functional, also existing for non-Lipschitz boundaries, we implement (in the two-dimensional case) the gradient descent method and find the optimal distribution with $50\%$ of the absorbent material on a frequency range with better performances than the $100\%$ absorbent boundary. The same type of performance is also obtained by the genetic method. 
\end{abstract}


\section{Introduction}

We study the parametric type optimization in the framework of boundary absorption of the total acoustical energy of the Helmholtz system in the case of a large class of boundaries, possibly irregular, non-rectifiable. The shape of the boundary is fixed, but it can be non-Lipschitz, such as fractal or multi-fractal. The optimization parameter is the characteristic function $\chi$ in the Robin boundary condition, defining the presence or absence of the boundary absorption (see~\eqref{EqHelmChi} and~\eqref{EqChi}). In this context (see Definition~\ref{DefPOP}), we show that the boundary irregularity, once the initial direct problem is weakly well-posed, does not affect the solution of the control problem and, in particular, the shape differentiability and continuity of the corresponding energy and of the solution of the direct problem. We solve it in the general class of domains, which we call here
\textit{Sobolev admissible domains}, consisting in bounded domains (open and connected sets) $\Omega\subset \R^n$, $n\ge 2$, which are 
\begin{enumerate}
	 \item[(a)] $(\eps,\infty)$-domains for a fixed $\eps>0$ (thus, $H^1$-extension domains);
	\item[(b)] its boundary $\del \Omega$ is the support of finite positive Borel $d$-upper regular measure for a fixed real number $d\in]n-2,n[$, \textit{i.e.} there exists $c_d>0$ such that:
\begin{equation}\label{Eqmu}
    \operatorname{supp}\mu=\del \Omega,\quad 
    \forall x\in\del \Omega,\ \forall r\in]0,1],\quad\mu(B_r(x))\le c_dr^d, 
\end{equation}
where $B_r(x)$ is an open ball of $\R^n$ centrerd at $x$ of radius $r>0$. 
\end{enumerate}
By $H^1$-extension domains, we understand the existence of a bounded linear extension operator $E_\Omega : H^1 (\Omega) \to H^1 (\R^n )$, \cite{HAJLASZ-2008,JONES-1981,ROGERS-2006}.
The $H^1$-extension domains are necessarily $n$-sets\cite{HAJLASZ-2008}, satisfying the measure density condition~\cite{HAJLASZ-2008}
$$\exists c>0 \quad \forall x\in \overline{\Omega}\quad  \lambda(B_r(x)\cap \Omega)\ge C  \lambda(B_r(x))=c r^n,$$
where $\lambda$ is the Lebesgue measure. In other words, the extension domains do not have a collapsing or infinitely fine boundary (such as cusps or fractal trees), so it is not possible to include a nontrivial open ball.

For the uniform geometrical properties of such domains, valuable for the uniform \\boundness of the extension operators and then for the uniform control of the direct problem solution, we work only with the particular case of $H^1$-extension domains:   $(\eps,\infty)$-domains for a fixed $\eps>0$. 
We refer to~\cite{JONES-1981,WALLIN-1991} for the definition of $(\eps,\infty)$-domain:
\begin{definition}\label{DefEDD}
	Let $\varepsilon > 0$. A bounded domain $\Omega\subset \mathbb{R}^n$ is called an \emph{$(\varepsilon,\infty)$-domain} if for all $x, y \in \Omega$ there is a rectifiable curve $\gamma\subset \Omega$ with length $\ell(\gamma)$ joining $x$ to $y$ and satisfying
	\begin{enumerate}
		\item[(i)] $\ell(\gamma)\le \frac{|x-y|}{\varepsilon}$ and
		\item[(ii)] $d(z,\partial \Omega)\ge \varepsilon |x-z|\frac{|y-z|}{|x-y|}$ for $z\in \gamma$.
	\end{enumerate}
\end{definition}
All bounded $(\eps,\infty)$-domains are $H^1$-extension domains of $\R^n$, $n\ge 3$, and they provide the optimal class of $H^1$-extension domains~\cite{JONES-1981} in $\R^2$, which are also the NTA-domains~\cite{NYSTROM-1996}. Thanks to~\cite[Theorem~1]{JONES-1981}, 
there is an extension operator $E_\Omega$ whose operator norm depends only on $n$ and $\eps$, see also~\cite{AZZAM-2017,ROGERS-2006}. As the definition of a Sobolev admissible domain depends on the domain $\Omega$ and the boundary measure $\mu$, we write the pair $(\Omega,\mu)$ to denote it. Condition~\eqref{Eqmu} implies that the Hausdorff dimension of the boundary is non-inferior to $d$. Thus, for Sobolev admissible domains, it allows having a boundary of variable dimension, for example, a boundary with a part defined by a $d_1$-set~\cite{JONSSON-1984}, a part with a Lipschitz boundary, and another part given by a $d_2$-set, with $d\le d_1<d_2< n$ and $n-2<d< n$.

The practical interest of the considered optimization problem is the concept of the anechoic chamber and the absorbing anti-noise walls with a fixed (small) quantity of the absorbing (porous) material, included in a reflective one, providing better (or at least the same) acoustical energy absorption performances as in the case of fully absorbing shapes (with the chosen porous material). It means that for a fixed boundary shape (and the boundary measure $\mu$),  we search the optimal distribution of the porous material inclusions (of a total small fixed volume) in the reflective one, able to minimize the total acoustical energy of the system on the frequency range of interest. For different applications and different source noises, the frequency range is different.
Globally, this question, \textit{``is it possible to have an optimal distribution of a small fixed quantity of porous material for at least the same energy absorption performances?'',} is a typical engineering problem since porous media are generally much more expensive and more complicated to produce than reflective ones. 
To our knowledge, this question has never been solved previously. 
However, the theoretical and numerical parametric shape optimization with different application goals is generally a very common subject as presented in~\cite{ALLAIRE-2007,ALLAIRE-1997,BANICHUK-1990,HASLINGER-2003,HENROT-2005,PIRONNEAU-1984}  and their references. For the geometrical shape optimization ($i.e.$ for the optimization of the boundary shape itself) for models with Robin-type conditions, we mainly refer to~\cite{BUCUR-2005,BUCUR-2016,HINZ-2021-1,HINZ-2023,MAGOULES-2021}.

In this article, we start with the question of the existence of at least one optimal distribution (see Definition~\ref{DefPOP}) for two typical cases: for a fixed frequency and then for a fixed range of frequencies. To conclude, we apply the usual relaxation method and the continuity energy property. Our main result is given in Theorem~\ref{ThMainROptParam}. Each time, we refer to the well-posedness of the direct problem presented in Section~\ref{SecWP}.
We denote by $\Gamma$ the Robin part of the boundary, supposed to be all times not trivial part of $\del \Omega$, and by $\Gamma_D$ the Dirichlet part. 
We distinguish two cases:  $\del \Omega=\Gamma$ (or $\mu(\Gamma)=\mu(\del \Omega)$, the typical case of anechoic chambers), and the case when $\Gamma_D$ is also a not trivial ($\mu(\Gamma_D)>0$) part of $\del \Omega$. 
The inclusion $\Gamma_D$ with the homogeneous Dirichlet condition previously was crucial for the use of the Poincaré inequality with a uniform on the shape boundary constant, only depending on its volume~\cite{MAGOULES-2021,HINZ-2021,HINZ-2021-1}. Thanks to~\cite[Theorem~3.1]{HINZ-2023}, its generalization was obtained for the case $\del \Omega=\Gamma$ for Sobolev admissible domains. 
Thus, the results of~\cite{HINZ-2023} make it possible to solve the direct and the parametric shape optimization problem for the Helmholtz equation, also in the pure Robin boundary case.

In Theorem~\ref{ThMainROptParam}, we prove the existence of an optimal distribution minimizing the energy of the relaxation problem (posed in the class of the shape admissible domains~\eqref{EqUad*}) globally (see~\eqref{EqExistFreqMin}), on a range of frequencies, and locally (see~\eqref{EqExistMin}), for a fixed frequency.
 In addition, we find the shape derivative of the energy, understood as usual in the Frechet sense, without any additional assumption on the regularity of the boundary, described by a $d$-upper regular measure. The value of the derivative obviously depends on the chosen boundary measure $\mu$ (see~\eqref{EqJder} and~\eqref{EqDerJhat}).

In Section~\ref{SecNumR}, we finish by giving a numerical example when the answer to the posed previously engineer question is positive with $50\%$ of reduction of the porous material (see Fig.~\ref{fig:comparaison_gen} which shows the analogous energy performances obtained by two different numerical methods: the gradient descent method, explained in Subsection~\ref{ss:algo_gradient}, and the genetic algorithm based on the covariance matrix adaptation evolution strategy~\cite{AHAMED-2013}).

The rest of the paper is organized in the following way: in Section~\ref{SecWP}, we introduce the model, the main notations and formulate in Theorem~\ref{ThWPchi} the well-posedness result known from~\cite{HINZ-2021-1}; in Section~\ref{SecParam} we introduce the parametric optimization problem and prove the main result stated in Theorem~\ref{ThMainROptParam}, firstly by proving the existence of optimal shapes in Subsection~\ref{SubsExistenceShape} and then by finding the shape derivatives of the energy functional by the usual Lagrangian method in Section~\ref{SecLagrange}; in Section~\ref{SecNumR} we present the numerical results found for the optimization on a fixed frequency (see Subsection~\ref{SubSecOptNumSF}) and for the optimization on a fixed frequency range (see Subsection~\ref{SubSecOptNumRF}), starting by describing the used gradient descent method in Subsection~\ref{ss:algo_gradient}.

\section{Model and its well-posedness properties}\label{SecWP}
We study the same frequency model as in~\cite{MAGOULES-2021}. We assume that $\Omega\subset \R^n$ and $\mu$ is the boundary measure such that $(\Omega,\mu)$ is a (bounded) Sobolev admissible domain, hence, with a compact boundary $\del \Omega$. To describe the energy absorption of a wave, we model it by the complex-valued Robin-type boundary condition with a coefficient $\alpha(f)$, a continuous function of the frequency $f$ (supposed to be real). In~\cite{MAGOULES-2021} it was mentionned that for the convection $\hat{w}(x,t)=e^{-i k t}w(x)$ with $k>0$, the absorbing properties of the porous media are modeled by $\operatorname{Im}(\alpha)<0$ and the reflective ones by $\operatorname{Re}(\alpha)>0$ (for the time-depending model see~\cite{BARDOS-1994}). Here,  we adopt these assumptions and consider the following mixed boundary-value problem for a fixed (real) wave number $k:=k(f)=\frac{2\pi f}{c}$ with a constant speed of the wave propagation in the air $c>0$: 
\begin{equation}\label{EqHelmChi}
  \left\{
    \begin{array}{ll}
        (\Delta+k^2) u=F  \mbox{ on }  \Omega; \\
        \frac{\partial u}{\partial n} = 0  \mbox{ on }    \Gamma_{Neu};\quad
         u = 0 \mbox{ on } \Gamma_{Dir};\quad
        \frac{\partial u}{\partial n}+\alpha(k)\chi u = \eta \mbox{ on } \Gamma,
    \end{array}
\right.  
\end{equation}
where  $\chi$ is the characteristic function of  the porous material inclusions on $x\in \Gamma$:
\begin{equation}\label{EqChi}
		\chi(x)=\left\{ \begin{array}{l}
	                                          	1, \hbox{ there is a porous material in  } x\\
	                                          	0, \hbox{ no porous material in } x.
	                                          \end{array}\right.
 \end{equation}
 To avoid degenerate cases,  we suppose that each part of $\Gamma$, porous and no porous, has positive capacity with respect to the space $H^1(\mathbb{R}^n)$ (see for instance~\cite[Section 7.2]{MAZ'JA-1985}) and has a strictly positive value of the measure $\mu$. Up to a zero $\mu$-measure set, the part of $\Gamma$ filled with the porous material can be considered its compact subset.  
 The partition of the boundary $\del \Omega= \Gamma_{Neu} \cup \Gamma_{Dir}\cup \Gamma$ is done with the same strategy as in~\cite{MAGOULES-2021} - the Dirichlet part for the noise source projected on the boundary, the Neumann part for the only reflexive parts, Robin condition for the part of the boundary consisting on two media, absorbing and non-absorbing one, - and the same meaning as in~\cite{HINZ-2021}:
 \begin{equation}\label{EqDivisionDelOmega}
 	\mu(\Gamma_{Neu}\cap \Gamma_{Dir})=\mu (\Gamma_{Neu}\cap \Gamma)=\mu(\Gamma_{Dir}\cap \Gamma)=0,
 \end{equation}
 $\Gamma_{Dir}$ and $\Gamma$ are closed subsets of $\del \Omega$. The assumptions that $\Gamma_{Dir}$, $\Gamma$ and its porous part are closed in the induced topology on $\del \Omega$ ensure that the linear trace operators $Tr_{\Gamma_{Dir}}: H^1(\Omega)\to L^2(\Gamma_{Dir},\mu)$ and $Tr_{\Gamma}: H^1(\Omega)\to L^2(\Gamma,\mu)$ 
 are compact (for their definitions see~\cite{HINZ-2021,HINZ-2021-1,HINZ-2023} initially adopted from~\cite[Corollaries 7.3 and 7.4]{BIEGERT-2009} and based on the restriction of  quasi-continuous representatives of $H^1(\R^n)$-elements).  
 
 Here, the space $L^2(\Gamma,\mu)$ means the space of measurable functions on $\Gamma$ such that $\|h\|_{L^2(\Gamma,\mu)}=\sqrt{\int_\Gamma |h|^2d \mu}$ is finite.
 
 The basic properties of the trace operator are presented in~\cite[Corollary~5.2]{HINZ-2021}. For more properties, see also the incoming work~\cite{CLARET-2023}.
 \begin{theorem}
 	Let $(\Omega,\mu)$ be a Sobolev admissible domain of $\R^n$, $n\ge2$.
Then the image of the trace operator $B(\del \Omega):=\operatorname{Tr}(H^1(\Omega))$ endowed with the norm
\begin{equation}\label{normTri}
\left\|h\right\|_{B(\del \Omega)}:=\min\{ \left\|v\right\|_{H^1(\Omega)}|\ h=\mathrm{Tr}\:v\},
\end{equation}
is a Hilbert space, dense and compact in $L^2(\del \Omega,\mu)$.
 \end{theorem}
We denote by $B'(\del \Omega)$ the topological dual space of $B(\del \Omega)$ and take in mind the usual Gelfand triple: $B (\del \Omega)\subset L^2(\del \Omega,\mu) \subset B'(\del \Omega)$.
 Therefore, the normal derivative on a boundary of a Sobolev admissible domain (possibly non-Lipschitz or fractal)
 is understood as the linear continuous functional on the image of the trace, defined by the usual Green formula:
 for all  $u\in H^1(\Omega)$ with $\Delta u\in L^2(\Omega)$ and for all $v\in H^1(\Omega)$
 \begin{equation}\label{EqGreenF}
 	\langle\frac{\del u}{\del n},\Tr v\rangle_{B'\!,\,B}= \int_\Omega{(\Delta u)v\,\dx}+\int_\Omega{\nabla u\cdot\nabla v\,\dx}.
 \end{equation}

 For the frequency model~\eqref{EqHelmChi}, we introduce the Hilbert space 
 \begin{equation}\label{EqVOM}
  V(\Omega)=\{u\in H^1(\Omega)|\; \operatorname{Tr}_{\Gamma_{Dir}} u=0 \}
 \end{equation}
with the norm (equivalent to the canonical norm $\|\cdot\|_{H^1(\Omega)}$ by the Poincaré inequality and the continuity of the trace operator)
\begin{equation}\label{EqVOmchi}
	\|u\|^2_{V(\Omega),\chi}=\int_\Omega |\nabla u|^2\dx+\int_\Gamma \mathrm{Re}(\alpha)\chi |u|^2 d\mu.
\end{equation}
If $\Gamma=\del \Omega$ then $V(\Omega)=H^1(\Omega)$ and  norm~\eqref{EqVOmchi} is still equivalent to the canonical norm of $H^1(\Omega)$ by~\cite[Corollary~3.2]{HINZ-2023} with the assumption of the lower and upper boundness of the coefficient $\mathrm{Re}(\alpha)\chi$ (see also~\cite[Corollary~5.15]{HINZ-2021-1}).

Norm~\eqref{EqVOmchi} is associated with the following inner product:
\begin{equation}\label{EqScPVchi}
\forall u,v\in V(\Omega) \quad	(u,v)_{V(\Omega),\chi}=\int_\Omega \nabla u \nabla \overline{v}\dx+\int_\Gamma \mathrm{Re}(\alpha)\chi \mathrm{Tr} u \mathrm{Tr}\overline{v} d \mu.
\end{equation}
The choice of $\chi$ from~\eqref{EqChi} comes from the physical meaning of our model. For the well-posedness, we do not need to assume that $\chi$ is a characteristic function, but  that $\chi\in L^\infty(\Gamma,\mu)$ is a nonnegative and bounded Borel function on $\Gamma$ which is positive with a positive minimum on a subset positive $\mu$-measure. 
For all such $\chi_1$, $\chi_2\in L^\infty(\Gamma,\mu)$ 
the norms $\|\cdot\|_{V(\Omega),\chi_1}$ and $\|\cdot\|_{V(\Omega),\chi_2}$ are equivalent on $V(\Omega)$. They are also equivalent by~\cite[Corollary 5.15]{HINZ-2021-1} to the canonical $H^1(\Omega)$- (thanks to the compactness of the trace operator) and  $H^1_0(\Omega)$- norms (as soon as $\mu(\Gamma_D)>0$). By the canonical $H^1_0(\Omega)$-norm we understand~\eqref{EqVOmchi} with   $\chi=0$ on $\Gamma$.

 In this case~\cite{HINZ-2021,MAGOULES-2021}, we search the weak solution $u\in V(\Omega)$ of system~\eqref{EqHelmChi}   in the following variational sense
\begin{multline}\label{EqVFchi}
	\forall \phi \in V(\Omega) \quad \int_\Omega \nabla u\nabla \overline{\phi} \dx -k^2 \int_\Omega u \overline{\phi} \dx+ \int_\Gamma \operatorname{Re}(\alpha) \chi \mathrm{Tr} u \mathrm{Tr}\overline{\phi}d \mu\\ +i  \int_\Gamma \operatorname{Im}(\alpha) \chi \mathrm{Tr} u \mathrm{Tr}\overline{\phi}d \mu =\int_\Gamma  \eta \mathrm{Tr}\overline{\phi}d \mu-\int_\Omega F \overline{\phi} \dx,
\end{multline}
for some fixed $F\in L^2(\Omega)$, $\chi\in L^\infty(\Gamma,\mu)$, $\eta\in L^2(\Gamma,\mu)$, $k>0$ and $\alpha\in \C$ with $\operatorname{Re}(\alpha)>0$, $\operatorname{Im}(\alpha)<0$ (constants on $\Gamma$, depending continuously on $k$). Here, the notation $\overline{\phi}$ means the complex conjugate of a complex-valued $\phi$.
Equivalently, the variational formulation~\eqref{EqVFchi} can be rewritten as
\begin{multline}
\forall \phi\in V(\Omega) \quad	(u,\phi)_{V(\Omega),\chi}-k^2(u,\phi)_{L^2(\Omega)}+i(\mathrm{Im}(\alpha) \chi \mathrm{Tr} u, \mathrm{Tr}\phi)_{L^2(\Gamma,\mu)}\\
=(\eta,\mathrm{Tr}\phi)_{L^2(\Gamma,\mu)}-(F,\phi)_{L^2(\Omega)}.
\end{multline}

We have the analogous well-posedness result for~\eqref{EqVFchi} to compare to~\cite[Theorem 2.1]{MAGOULES-2021}:
\begin{theorem}\label{ThWPchi}
	Let $\Omega\subset \R^n$ and $\mu$ be the boundary measure such that $(\Omega,\mu)$ is a  Sobolev admissible domain with a (compact) boundary $\del \Omega$. 
Assume $\del \Omega=\Gamma_{Dir}\cup\Gamma_{Neu}\cup \Gamma$ such that it holds~\eqref{EqDivisionDelOmega}, $\mu(\Gamma)>0$, $\Gamma\subseteq \del \Omega$ and $\Gamma_{Dir}$ and $\Gamma$ are also  compact with the same properties as $\del \Omega$ itself.
Let in addition $\mathrm{Re}(\alpha)>0$, $\mathrm{Im}(\alpha)<0$  on~$\Gamma$ ($\alpha$ is a continuous function or simply a constant) and  $\chi\in L^\infty(\Gamma,\mu)$ be a nonnegative and bounded Borel function on $\Gamma$ which is positive with a positive minimum on a subset positive $\mu$-measure. 

Then for all $F\in L^2(\Omega)$,  $\eta\in L^2(\Gamma,\mu)$, and $k>0$ (or equivalently, $f>0$) there exists 
a unique solution $u\in V(\Omega)$ of the Helmholtz problem~\eqref{EqHelmChi} 
in the following sense: for all $v\in V(\Omega)$ it holds~\eqref{EqVFchi}.

Moreover, the solution of problem~\eqref{EqVFchi} $u\in  V(\Omega)$, continuously depends on the data:
there exists a constant $\hat{C}>0$, depending only on $n$, $\eps$, $d$, $c_d$, $\lambda(\Omega)$, $\chi$, $\alpha$ and $k$,  
such that
 \begin{equation}\label{EqApriorichi}
  \|u\|_{V(\Omega),\chi}\le \hat{C}\left(\|F\|_{L_2(\Omega)}+\|\eta\|_{L^2(\Gamma,\mu)}\right).
 \end{equation}
\end{theorem}
The proof of Theorem~\ref{ThWPchi} is completely analogous to the proof of~\cite[Theorem 2.1]{MAGOULES-2021} and~\cite[Theorem 7.1]{HINZ-2021} (see Appendix B in~\cite{ARXIV-HINZ-2020}) by applying the Fredholm alternative and the uniqueness of the homogeneous Cauchy problem for $-\Delta+k^2$ (which ensures the injective property). The important remark is the application of the Poincaré inequality with the uniform constant $C_P>0$ depending only on the fixed parameters $n$, $\eps$, $d$, $c_d$ in the case of $\Gamma=\del \Omega$ following~\cite[Theorem~3.1.]{HINZ-2023}. In the case $\mu(\Gamma_D)\ne 0$ we apply~\cite[Theorem~10]{DEKKERS-2022}. 
\begin{remark}\label{RemOptT}
In what follows, we use this established Fredholm property of problem~\eqref{EqVFchi}, especially for $F=0$.
More precisely, thanks to Appendix B in~\cite{ARXIV-HINZ-2020} and also to~\cite[Theorem~2.1]{MAGOULES-2021}, the variational formulation~\eqref{EqVFchi} for $F=0$ can be presented
in the following operator form:
\begin{equation}\label{EqOpT}
((Id-k^2T)u,\phi)_{V(\Omega),\chi}=(\eta, \mathrm{Tr}\phi)_{L^2(\Gamma,\mu)},	
\end{equation}
where $Id-k^2T$ is bijective operator on $V(\Omega)$ with $Id$ the identity operator and $T$ a compact operator. The compactness of $T$ follows from the compactness of the trace operator (see~\cite{BIEGERT-2009},~\cite[Theorem~5.10]{HINZ-2021-1},~\cite[Theorem~2.1]{HINZ-2023}) and the embedding $H^1(\Omega)\subset L^2(\Omega)$ (the compactness holds for any bounded extension domain~\cite{ARFI-2019,ROZANOVA-PIERRAT-2021,HINZ-2021-1}). As it was also mentioned in~\cite{BARDOS-1994}, the wave number $k$ is real and thus not in the spectrum of $-\Delta$ associated with the absorption Robin boundary condition with $\operatorname{Im}\alpha \ne 0$ on a non-trivial part of the boundary.
\end{remark}

\section{Parametric optimization problem on the support of a $d$-upper regular measure}\label{SecParam}
Once we have the well-posedness of our model, let us formulate the optimization problem.
We fix the volume fraction $\beta$ of the absorbing material of the wall (a number between $0$ and $1$, in the assumption that $\mu(\Gamma)=1$): 
\begin{equation}\label{Eqbeta}
	0<\int_\Gamma \chi d\mu =\beta <\mu(\Gamma)=\int_\Gamma 1 d \mu=1.
\end{equation}
It is the percentage rate of the absorbent material on $\Gamma$.
The two limit cases $\beta=0$ and $\beta=1$ are excluded.
Therefore, we define the space of admissible distributions of the porous material $\chi$:
\begin{multline}\label{EqUadChi}
	U_{ad}(\beta)=\{\chi\in L^\infty(\Gamma,\mu)|\; \mu\hbox{-a.e. on } \Gamma \quad \chi(x)\in\{0,1\}, \quad 0<\beta= \int_\Gamma \chi d\mu< 1\}.
\end{multline}
The set $U_{ad}(\beta)$ is thus a subset of $L^\infty(\Gamma,\mu)$ consisting of all functions taking only two values $0$ or $1$ on $\Gamma$ and define a fixed value of $\beta$ in $]0,1[$. Let $I$ be the frequency interval of interest (for instance, the audible frequencies).
We want to minimize the total acoustical energy of problem~\eqref{EqHelmChi} first for a fixed frequency $f\in I$ and then for the all frequency interval $I$. For a fixed frequency $f\in I$, the acoustical energy is modeled by the following functional $J (f,\chi):I\times U_{ad}(\beta)\to \R$:
\begin{equation}\label{EqJ}
	 J(f,\chi)=A\int_\Omega |u(f,\chi)|^2\dx+B\int_\Omega |\nabla u(f,\chi)|^2 \dx+C\int_\Gamma |\operatorname{Tr}u(f,\chi)|^2 d\mu
\end{equation}
with positive constants $A\ge 0$, $B\ge 0$ and $C\ge0$, $A^2+B^2>0$. If $A\ge0$, and  $B$ with $C$ are strictly positive, the expresion of $J$ defines an equivalent norm on $H^1(\Omega)$, and hence, on $V(\Omega)$. For the numerical tests in Section~\ref{SecNumR}, we take $A=1$ and $B=C=0$. Hence, the right-hand side of~\eqref{EqJ} presents the $L^2$-norm of the weak solution of~\eqref{EqHelmChi}. The changes of the distribution $\chi$, keeping fixed all other data and parameters of system~\eqref{EqHelmChi}, imply the changes of the weak solution $u$, and thus also vary the value of $J$. Therefore, in~\eqref{EqJ} we consider $u$ as a function of $\chi$ and $J$ depending on $\chi$ by $u$. In the same way, the dependence of $u$ and $J$ on the frequency value $f$ is also obvious.
Therefore, our final aim is to minimize the ``total'' energy on $U_{ad}(\beta)$:
\begin{equation}\label{EqTotalJ}
	\hat{J}(\chi):=\int_I J(f,\chi) d f, \quad \min_{\chi\in U_{ad}(\beta)}\hat{J}(\chi).
\end{equation}

Thus, we formulate two optimization problems:
\begin{definition}\label{DefPOP}\textbf{(Parametric optimization problems)}
	For a fixed  Sobolev admissible domain $(\Omega,\mu)$, fixed $\beta>0$, and the source of the noise $(F,\eta)$ and the chosen porous material described by $\alpha(f)$, a known function of the frequency $f$ with $\mathrm{Re}(\alpha)>0$, $\mathrm{Im}(\alpha)<0$  on~$\Gamma$,
	\begin{enumerate}
		\item \textbf{for a fixed frequecy} $f>0$, to find $\chi_{opt}\in U_{ad}(\beta)$ for which there exists the (unique) solution $u(f,\chi_{opt})\in V(\Omega)$ of the Helmholtz problem~\eqref{EqHelmChi} considered with $\chi=\chi_{opt}$, such that
$$J(f,\chi_{opt})=\min_{\chi\in U_{ad}(\beta)}J(f,\chi).$$
\item \textbf{for a bounded range of frequencies} $I$, to find $\chi_{opt}\in U_{ad}(\beta)$ for which there exists for all $f\in I$ the (unique) solution $u(f,\chi_{opt})\in V(\Omega)$ of the Helmholtz problem~\eqref{EqHelmChi} considered with $\chi=\chi_{opt}$, such that
$$\hat{J}(\chi_{opt})=\min_{\chi\in U_{ad}(\beta)}\int_I J(f,\chi)d f.$$
	\end{enumerate}
\end{definition}

The main problem is that the set of the admissible shapes $U_{ad}(\beta)$ is not closed for the weak$^*$ convergence of $L^\infty(\Gamma,\mu)$~\cite{HENROT-2005}:
if a sequence of characteristic functions $(\chi_n)_{n\in \N}$ converges weakly$^*$ in $L^\infty(\Gamma,\mu)$ to a function $h\in L^\infty(\Gamma,\mu)$, it does not follows  that the weak$^*$ limit function $h$ is a characteristic function, $i.e.$ takes only two values $0$ and $1$. 

Thus $U_{ad}(\beta)$ is not a weak$^*$ compact. Consequently, the parametric shape optimization problem defined in Definition~\ref{DefPOP} cannot be generally solved on $U_{ad}(\beta)$.

Following the standard relaxation approach~\cite[p.277]{HENROT-2005}, instead of solving the optimization problem on $U_{ad}(\beta)$ we will solve it on its (convex) closure (for the ideas of the proof see~\cite[Proposition 7.2.14]{HENROT-2005}):
\begin{equation}\label{EqUad*}
	U_{ad}^*(\beta)=\{\chi\in L^\infty(\Gamma,\mu)|\; \mu\hbox{-a.e. on }\Gamma \quad \chi(x)\in [0,1], \quad 0<\int_{\Gamma}\chi d \mu =\beta < \mu (\Gamma)=1\}. 
\end{equation}
Let us notice that  $\|\chi\|_{L^\infty(\Gamma,\mu)}=1$ for all $\chi\in U_{ad}(\beta)$, while 
for all $\chi\in U_{ad}^*(\beta)$ it holds 
\begin{equation}\label{EqBLInfNorm}
	0<\beta\le \| \chi\|_{L^\infty(\Gamma,\mu)}\le 1.
\end{equation}

Therefore, we have, in the same way as in~\cite[Proposition 7.2.14]{HENROT-2005},
\begin{theorem}\label{ThU*}
	Let $\beta\in ]0,1[$ be fixed. If $U^*_{ad}(\beta)$ is given by~\eqref{EqUad*}, then
	$U^*_{ad}(\beta)$ is the weak$^*$ closed convex hull of $U_{ad}(\beta)$ and $U_{ad}(\beta)$ is exactly
the set of extreme points of the convex set $U^*_{ad}(\beta)$. 
	\end{theorem}
We denote by $J^*: I\times U^*_{ad}(\beta)\to \R$ the extended functional (with, as previously for $J$,  positive constants $A\ge 0$, $B\ge 0$ and $C\ge0$, $A^2+B^2>0$ and a fixed frequecy $f\in I$)
\begin{equation}\label{EqJ*}
	\forall \chi\in U^*_{ad}(\beta) \; J^*(f,\chi)=A\int_\Omega| u(f,\chi)|^2\dx+B\int_\Omega |\nabla u(f,\chi)|^2 \dx+C\int_\Gamma |\operatorname{Tr} u(f,\chi)|^2 d\mu,
\end{equation}
which in addition satisfies $J^*(f,\chi)|_{U_{ad}(\beta)}=J(f,\chi)$. 
Here, $u(f,\chi)$ is the weak solution of system~\eqref{EqHelmChi} found for a chosen $(f,\chi)$.
We also denote
\begin{equation}\label{EqhatJ*}
	\forall \chi\in U^*_{ad}(\beta) \quad \hat{J}^*(\chi)=\int_I J^*(f,\chi) d f,
\end{equation}
satisfying $\hat{J}^*(\chi)|_{U_{ad}(\beta)}=\hat{J}(\chi)$.

To solve the parametric optimization problem on $U^*_{ad}(\beta)$ we need to ensure that the constant $\hat{C}$ in estimate~\eqref{EqApriorichi} does not depend on $\chi$, when $\chi\in U^*_{ad}(\beta)$. 
If $\mu(\Gamma_{Dir})>0$, then it follows from the upper uniform boundness of the $L^\infty$ norm of all $\chi$ on $U^*_{ad}(\beta)$ (see~\eqref{EqBLInfNorm}) and the equivalence of norms with uniform on $\chi$ constants:   for all $\chi \in U^*_{ad}(\beta)$ 
there exist $C_0>0$ independent on $\chi \in U^*_{ad}(\beta)$ such that
\begin{equation}\label{EqNormEq}
	\forall v\in V(\Omega) \quad \|v\|_{H^1_0(\Omega)}\le \|v\|_{V(\Omega),\chi}\le C_0 \|v\|_{H^1_0(\Omega)}.
\end{equation}
For the proof, it is sufficient to apply the continuity of the trace operator and the Poincaré inequality ($\Omega$ is  fixed  in our framework), and to obtain $$C_0=1+\|\alpha_R\|_{C(\Gamma)}C(\|\mathrm{Tr}\|_{\mathcal{L}(V(\Omega),L^2(\Gamma,\mu))})C_P(\lambda(\Omega)),$$ independent on $\chi$. 

In the case of $\Gamma=\del \Omega$, we need to use~\cite[Corollary~3.2]{HINZ-2023} with the uniform on $\chi$ constants for the norm equivalences on $H^1(\Omega)$ (with $\|\cdot\|_{H^1(\Omega)}$ in~\eqref{EqNormEq} instead of $\|\cdot\|_{H^1_0(\Omega)}$). Therefore, we have to add the assumption that for all $\chi \in U^*_{ad}(\beta)$, $\chi\ge \chi_{min}>0$ on $\operatorname{supp}\chi$ for a fixed uniform constant $\chi_{min}>0$. In other words, for the case $\Gamma=\del \Omega$, instead of $U_{ad}(\beta)$ and $U_{ad}^*(\beta)$, we consider  $U_{ad}(\beta,\chi_{\min})$ and $U_{ad}^*(\beta,\chi_{\min})$ respectively. In what follows, we will however use only the notation $U_{ad}(\beta)$ and $U_{ad}^*(\beta)$, which should be understood with this corrective uniform lower boundness by $\chi_{min}>0$ condition for the pure Robin case $\Gamma=\del \Omega$.
\begin{lemma}\label{LemIndepChi}
	Let $\beta\in ]0,1[$ (for $\mu(\Gamma)=1$) be fixed and all assumptions of Theorem~\ref{ThWPchi} hold. Then for all $\chi\in U^*_{ad}(\beta)$, there exists a constant $\hat{C}^*>0$, depending only on  $\alpha$, $k$ and on $C_P$ (the Poincaré uniform constant depending only on $\eps$, $n$, $d$, $c_d$ and $\lambda(\Omega)$), \textit{but not on $\chi$}, such that estimate~\eqref{EqApriorichi} holds for the corresponding weak solution of~\eqref{EqHelmChi} on the fixed Sobolev admissible domain $(\Omega,\mu)$. 
\end{lemma}

Now we state our main result:
\begin{theorem}\label{ThMainROptParam}
	Let $\Gamma$, $F$, $\eta$, $f=f_0\in I$, $\alpha(f_0)$ be fixed in a way that all assumptions of Theorem~\ref{ThWPchi} are satisfied on a fixed Sobolev admissible domain $(\Omega,\mu)$ of $\R^n$ with $\mu(\Gamma)=1$.  Then
	for a fixed $\beta\in ]0, 1[$,  there exists (at least one) optimal distribution $\chi^{opt}\in U^*_{ad}(\beta)$ and the corresponding optimal solution $u(f_0,\chi^{opt})\in V(\Omega)$ of system~\eqref{EqHelmChi}, such that
	\begin{equation}\label{EqExistMin}
	J^*(f_0,\chi^{opt})=\min_{\chi \in U^*_{ad}(\beta)} J^*(f_0,\chi)=\inf_{\chi \in U_{ad}(\beta)} J(f_0,\chi),	
	\end{equation}
	and there exists $\hat{\chi}^{opt}\in U^*_{ad}(\beta)$ such that on a fixed bounded plage of frequencies $I\subset \R^{+ *}$
	\begin{equation}\label{EqExistFreqMin}   
		\hat{J}^*(\hat{\chi}^{opt})=\min_{\chi \in U^*_{ad}(\beta)} \hat{J}^*(\chi)=\inf_{\chi \in U_{ad}(\beta)} \hat{J}(\chi).		
	\end{equation}

	In addition, the functional $J^*$ is Fréchet differentiable on $\chi\in U^*_{ad}(\beta)$. Its directional derivative in $\chi\in U^*_{ad}(\beta)$ in the direction $h\in L^\infty(\Gamma,\mu)$ with $\int_\Gamma h d \mu =0$ and such that $\chi+h\in U^*_{ad}(\beta)$  is given by
	 \begin{equation}\label{EqJder}
	 	\langle D_\chi J^*(f_0,\chi),h\rangle=-\int_\Gamma h \mathrm{Re} (\alpha(f_0) u(f_0,\chi) p(f_0,\chi))d \mu,
	\end{equation}
	where $u(f_0,\chi)$ is the weak solution of~\eqref{EqHelmChi}  and $p(f_0,\chi)$ is the weak solution of the adjoint problem (with $k=\frac{2\pi f_0}{c}$):
	\begin{multline}\label{EqVFAdjProb}
\forall \phi\in V(\Omega) \quad	(p,\phi)_{V(\Omega),{\chi}}-k^2(p,\phi)_{L^2(\Omega)}+i(\mathrm{Im}(\alpha) \chi\mathrm{Tr} u, \mathrm{Tr}\phi)_{L^2(\Gamma,\mu)}\\
=(2A\overline{u}(f_0,\chi)-2B\Delta \overline{u}(f_0,\chi),\phi)_{L^2(\Omega)}+(2C \overline{u}(f_0,\chi)-2B\chi\overline{\alpha}\overline{u}(f_0,\chi),\mathrm{Tr} \phi)_{L^2(\Gamma,\mu)}.
\end{multline}
	 Finally, the Frechet derivative on $\chi\in U^*_{ad}(\beta)$ of $\hat{J}^*$ is given by
	 \begin{equation}\label{EqDerJhat}
	 	D_\chi \hat{J}^*(\chi)=D_\chi\left( \int_I J^*(f,\chi) d f\right)=\int_I D_\chi J^*(f,\chi) d f.
	 \end{equation}

\end{theorem}
\begin{remark}
In the source therm of~\eqref{EqVFAdjProb} $\overline{u}$ denotes the complex conjugate of $u$.
        The regularity $\Delta u\in L^2(\Omega)$ holds as in the distributional sense $\Delta u=-k^2 u+F\in L^2(\Omega)$.
        \end{remark}
\subsection{Existence of optimal shapes}\label{SubsExistenceShape}

To prove Theorem~\ref{ThMainROptParam}, we firstly show the continuity of the energy
for a fixed frequency $f_0\in I$ on $\chi\in U^*_{ad}(\beta)$ by the weak$^*$ topology. As the frequency is supposed to be fixed, we can simplify the notations a little bit by omitting $f_0$. Instead, in what follows, we explicitly write that the energy depends on the solution of the Helmholtz problem:
	$u(f_0,\chi)$ and  $J^*(f_0,\chi)$  are denoted by $u(\chi)$  and  $J^*(\chi,u(\chi))$  respectively.

\begin{lemma}\label{LemContWeak} Let  $f_0\in I$ be fixed.
	If $(\chi_j)_{j\in \N}\subset U^*_{ad}(\beta)$ such that 
	$\chi_j\stackrel{*}{\rightharpoonup} \chi \hbox{ in } L^\infty(\Gamma,\mu),$
	then $\chi \in U^*_{ad}(\beta)$, $u(\chi_j)\to u(\chi)$ in $V(\Omega)$  and for a constant $c=c(A,B,C)>0$
	\begin{multline}\label{EqJ*Cont}
		|J^*(\chi_j,u(\chi_j))-J^*(\chi,u(\chi))|\le c \left|\|u(\chi_j)\|^2_{V(\Omega)}-\|u(\chi)\|^2_{V(\Omega)}\right| \to 0 \quad j\to +\infty.
	\end{multline}
	\end{lemma}

\begin{proof}
	As $U^*_{ad}(\beta)$ is closed for the weak$^*$ convergence, the weak$^*$ limit $\chi$ of a sequence $(\chi_j)_{j\in \N}\subset U^*_{ad}(\beta)$ belongs to $U^*_{ad}(\beta)$.
	
	Let us denote by $u_j$ and $u$ the weak solutions of~\eqref{EqHelmChi} found for $\chi_j$ and $\chi$, respectively. 
		Thus, the difference $v_j=u_j-u$ is the weak solution of the following system:
	\begin{equation}\label{EqHelmDif}
  \left\{
    \begin{array}{ll}
        (\Delta+k^2) v_j=0  \mbox{ on }  \Omega, \\
        \frac{\partial v_j}{\partial n} = 0  \mbox{ on }    \Gamma_{Neu},\\
         v_j = 0 \mbox{ on } \Gamma_{Dir},\\
        \frac{\partial v_j}{\partial n}+\alpha\chi_j v_j = \alpha(\chi-\chi_j) u \mbox{ on } \Gamma.
    \end{array}
\right.  
\end{equation}
We apply Theorem~\ref{ThWPchi} for system~\eqref{EqHelmDif}, taking $F=0$ and  $\eta=\alpha(\chi-\chi_j) \mathrm{Tr}u$. It sufficient to notice that $\eta\in L^2(\Gamma,\mu)$, since $\alpha$ is a constant (or a continuous function), $\chi-\chi_j$ belongs to $L^\infty(\Gamma,\mu)$, and, as $u\in V(\Omega)$, as the weak solution of~\eqref{EqHelmChi} associated to $\chi$, its trace $\mathrm{Tr}u\in  L^2(\Gamma,\mu)$. 
Therefore, the result of Theorem~\ref{ThWPchi} holds:
for all $j\in \N$ there exists a unique solution $v_j\in V(\Omega)$.
Moreover, the sequence $(v_j)_{j\in \N}$ is bounded in $V(\Omega)$, $i.e.$
there exists a constant $c>0$ (independent on $j$) such that 
\begin{equation}\label{EqVkB}
	\forall j\in \N \quad \|v_j\|_{V(\Omega),\chi_j}\le c.
\end{equation}

Indeed, by estimate~\eqref{EqApriorichi} and Lemma~\ref{LemIndepChi}, there exists a constant $c>0$, uniform on $j$, such that
\begin{multline}\label{EqEstDiffu}
	\|v_j\|_{V(\Omega),\chi_j}\le c\|\alpha(\chi-\chi_j) \mathrm{Tr}u\|_{L^2(\Gamma,\mu)}\\\le c\|\alpha\|_{L^\infty(\Gamma,\mu)}\|\chi-\chi_j\|_{L^\infty(\Gamma,\mu)} \|\mathrm{Tr}u\|_{L^2(\Gamma,\mu)}.
\end{multline}
By our assumptions (see the assumptions of Theorem~\ref{ThWPchi}) $\alpha$ is continuous on $\Gamma$ and $\Gamma$ is compact, then $\|\alpha\|_{L^\infty(\Gamma,\mu)}$ is finite and does not depend on $j$.
As $\chi_j\stackrel{*}{\rightharpoonup} \chi \hbox{ in } L^\infty(\Gamma,\mu),$ then the sequence $(\chi-\chi_j)_{j\in \N}$ is bounded in $L^\infty(\Gamma,\mu)$. 
In addition $\|\mathrm{Tr}u\|_{L^2(\Gamma,\mu)}$ does not depend on $j$.
Therefore, we conclude that the sequence $(v_j)_{j\in \N}$ is bounded in $V(\Omega)$, and \eqref{EqVkB} holds.

As $V(\Omega)$ is a Hilbert space (hence reflexive), there exists a weakly convergent subsequence:
$$\exists (v_{j_i})_{i\in \N}\subset (v_j)_{j\in \N}\hbox{ and } v\in V(\Omega): \quad  v_{j_i}\rightharpoonup v \hbox{ in } V(\Omega).$$
Now, we need to show that $v=0$. Once again, we use the well-posedness of~\eqref{EqVFchi} and its Fredholm property recalled in Remark~\ref{RemOptT}. By~\eqref{EqOpT}, we have for all $j\in \N$ and all $\phi\in V(\Omega)$ that
\begin{equation}\label{EqNE}
	((Id-k^2T)v_{j_i},\phi)_{V(\Omega),\chi_{j_i}}=(\alpha(\chi-\chi_{j_i}) \mathrm{Tr}u, \mathrm{Tr}\phi)_{L^2(\Gamma,\mu)}.
\end{equation}
The operator $Id-k^2T$ is linear, continuous, and bijective on $V(\Omega)$. By its continuity, for $i\to+\infty$, $ (Id-k^2T)v_{j_i}\rightharpoonup (Id-k^2T)v$  in  $V(\Omega)$, and hence, in the limit~\eqref{EqNE} becomes
\begin{equation}\label{EqNE2}
	((Id-k^2T)v,\phi)_{V(\Omega),\chi}=0 \hbox{ for all } \phi\in V(\Omega).
\end{equation}
By the injectivity of the operator $Id-k^2T$, relation~\eqref{EqNE2} implies that $v=0$.

Let us now prove that $v_{j_i}\to 0$ in $V(\Omega)$ (strongly!). We take in~\eqref{EqNE} $\phi=v_{j_i}$ and notice that by the compactness of the trace operator $\mathrm{Tr}: V(\Omega)\to L^2(\Gamma,\mu)$, we have, as $v_{j_i}\rightharpoonup 0$ in  $V(\Omega)$, that 
	$\mathrm{Tr}v_{j_i} \to 0$   in $L^2(\Gamma,\mu).$ Therefore, 
	we find that 
	\begin{equation}\label{EqNE3}
	((Id-k^2T)v_{j_i},v_{j_i})_{V(\Omega),\chi_{j_i}}=(\alpha(\chi-\chi_{j_i}) \mathrm{Tr}u, \mathrm{Tr}v_{j_i})_{L^2(\Gamma,\mu)}\to 0 \hbox{ for } i\to+\infty.
\end{equation}
Consequently, $\left[\|v_{j_i}\|^2_{V(\Omega),\chi_{j_i}}-k^2(Tv_{j_i},v_{j_i})_{V(\Omega),\chi_{j_i}}\right]\to 0$ for $i\to+\infty$. By its definition the operator $T: V(\Omega) \to V(\Omega)$ is compact (see Remark~\ref{RemOptT}), and hence, $Tv_{j_i}\to 0$ in $V(\Omega)$, which result in $(Tv_{j_i},v_{j_i})_{V(\Omega),\chi_{j_i}}\to 0$ for $i\to +\infty$. Finally, we obtain that $\|v_{j_i}\|^2_{V(\Omega),\chi_{j_i}}\to 0$ and $v_{i_j}\rightharpoonup 0$ in  $V(\Omega)$, ensuring the strong convergence $v_{j_i} \to 0$ in~$V(\Omega)$.

Now,  any weakly convergent sequence of solutions of system~\eqref{EqHelmDif} converges to $0$, $i.e.$ $0$ is the unique accumulation (limit) point 
of the sequence $(v_j)_{j\in \N}$, which implies that the sequence $v_j\to 0$ itself in $V(\Omega)$. Actually, $v_j=u_j-u$ and hence $u_j\to u$ in $V(\Omega)$ for $j\to+\infty$.

Finally, we have obtained that from $\chi_j \stackrel{*}{\rightharpoonup} \chi \hbox{ in } L^\infty(\Gamma,\mu),$ it follows that  $\chi \in U^*_{ad}(\beta)$ and $u(f_0,\chi_j)\to u(f_0,\chi)$ in $V(\Omega)$, from where we directly have~\eqref{EqJ*Cont}, $i.e.$ the continuity of $J^*$ on $U_{ad}^*(\beta)$ for all fixed $f_0\in I$.
\end{proof}

Therefore, we proceed to the proof of the existence of optimal distributions satisfying~\eqref{EqExistMin} and~\eqref{EqExistFreqMin}, stated in Theorem~\ref{ThMainROptParam}.
	\begin{proof}
As $U^*_{ad}(\beta)$ is weakly$^*$ compact (in $L^\infty(\Gamma,\mu)$) and $J^*$ and $\hat{J}^*$ are continuous on it (by the continuity of $u(\cdot,\chi)$ and the definitions of $J^*$ and $\hat{J}^*$), then there exist  $\chi^{opt},$ $\hat{\chi}^{opt} \in U^*_{ad}(\beta)$ (and the corresponding solutions of the Helmholtz system~\eqref{EqHelmChi}) realizing the minima of $J^*$ and $\hat{J}^*$ respectively on $U^*_{ad}(\beta)$.
	In addition, $$\min_{\chi \in U^*_{ad}(\beta)} J^*(f_0,\chi)=\inf_{\chi \in U_{ad}(\beta)} J(f_0,\chi)$$ as $U^*_{ad}(\beta)$ is the closure of $U_{ad}(\beta)$ and $J^*$ takes the same values as $J$ on $U_{ad}(\beta)$ (see Theorem~\ref{ThU*}). In the same way, we conclude for $\hat{J}^*$.
\end{proof}

To find~\eqref{EqJder}, we apply the usual Lagrangian~\cite{ALLAIRE-2007} method breifly presented in Section~\ref{SecLagrange}. By the linearity of the integral over $f$ and the Fréchet derivative properties, we obtain directly~\eqref{EqDerJhat}.
\subsection{Shape derivative of the energy for a fixed frequency}\label{SecLagrange}
In this section, the frequency $f$ is supposed to be fixed, and we commit any notations with dependence on it.
We adopt the Lagrangian method~\cite{ALLAIRE-2007} for the complex value problem:
instead, to consider the variational formulation~\eqref{EqVFchi}, we first split it into the real and imaginary parts and then consider their linear combination, using the independence of the real and imaginary parts of the test functions. 
We define the notations $u=u_R+iu_I$ with $u_R$ and $u_I$ for the real and imaginary parts of $u$, respectively. We do the same for $\alpha=\alpha_R+i\alpha_I$ and $F=F_R+iF_I$, taking $\chi\in U^*_{ad}(\beta)$ and $\eta=0$ (see~\eqref{EqVFchi}).  
For instance, we subtract the imaginary part of~\eqref{EqVFchi} from the real one and obtain its real-valued analog: 
for  all $(v_R, v_I)\in V(\Omega)\times V(\Omega)$
\begin{multline*}
FV(\chi,u_R,u_I,v_R,v_I):=\\
\int_{\Gamma}\chi\left[(\Tr\alpha_{I}\Tr u_{R}+\alpha_{R}\Tr u_{I}) \Tr v_{I}-(\alpha_{R} \Tr u_{R}-\alpha_{I}\Tr u_{I})\Tr v_{R}\right]d\mu\\+\int_{\Omega}\left[\nabla u_{I}\nabla v_{I}
-\nabla u_{R} \nabla v_{R}+k^{2}(u_{R}v_{R}-u_{I}v_{I})+f_{I}v_{I}-f_{R}v_{R}\right]dx=0.
\end{multline*}
Here $u_R$ and $u_I$ depend of $\chi$.

    We now define the Lagrangian of the optimization problem as the sum of the previously calculated variational formulation and the objective function~\eqref{EqJ*}
    \begin{multline*}
    \forall \chi\in U_{ad}^*(\Gamma,\mu), \quad \forall w_R,\, w_I, \, q_R,\, q_I\in V(\Omega) \\   \mathcal{L}(\chi,w_{R},w_{I},q_{R}, q_{I})=A\int_{\Omega}(w_{R}^{2}+w_{I}^{2})dx+B\int_{\Omega}((\nabla w_{R})^{2}+(\nabla w_{I})^{2})dx\\+C\int_{\Gamma} ((\Tr w_{R})^{2}+(\Tr w_{I})^{2})d \mu
    +FV(\chi,w_R,w_I,q_R,q_I).
            \end{multline*}
        Here all arguments of the Lagrangian, $\chi$, $w_R,\, w_I, \, q_R,\, q_I$ are independent.
        Then, to be able to apply the Lagrangian method ensuring 
        \begin{equation}\label{EqDerivLagrChi}
        	\langle D_\chi J^*(\chi),h\rangle=\langle \frac{\del\mathcal{L}}{\del \chi} (\chi, u_R(\chi), u_I(\chi),p_R(\chi),p_I(\chi)), h\rangle
        \end{equation}
        in a direction $h\in L^\infty(\Gamma,\mu)$
for $p=p_R+ip_I\in V(\Omega)$ the weak solution of the adjoint variational problem, we need firstly prove the differentiability of $u$ on $\chi$:
        \begin{lemma}\label{LemDiffer}
        Let assumptions of Theorem~\ref{ThMainROptParam} hold. The mapping $\chi\mapsto u(\chi)$ is Fréchet differentiable on $U_{ad}^*(\beta)$ and the directional derivative in $\chi\in\mathcal{U}_{ad}^*(\beta)$ in the direction $h\in L^{\infty}(\Gamma,\mu)$, such that $\chi+h\in U_{ad}^*(\beta)$, is given by 
$\langle D_\chi u(\chi),h \rangle=\psi$, where $\psi\in V(\Omega)$ is
the unique 
weak solution of the following variational formulation:
\begin{multline}\label{EqVFWder}
\forall \phi\in V(\Omega) \quad	(\psi,\phi)_{V(\Omega),\chi}-k^2(\psi,\phi)_{L^2(\Omega)}+i(\mathrm{Im}(\alpha) \chi \mathrm{Tr} \psi, \mathrm{Tr}\phi)_{L^2(\Gamma,\mu)}\\
=-(\alpha h \mathrm{Tr} u(\chi),\mathrm{Tr}\phi)_{L^2(\Gamma,\mu)},
\end{multline}
formally associated with the problem
\begin{equation}\label{frechet_derivability}
\left\{
    \begin{array}{ll}
        (\Delta +k^2) \psi=0  \mbox{ in }  \Omega , \\
        \frac{\partial \psi}{\partial n} = 0  \mbox{ on }    \Gamma_{N} ,\quad
        \psi = 0 \mbox{ on } \Gamma_{D} ,\\
        \frac{\partial \psi}{\partial n}+\chi\alpha \psi = -\alpha h u(\chi) \mbox{ on } \Gamma.
    \end{array}
\right.
\end{equation}
\end{lemma}

\begin{proof} We follow the ideas of the proof of~\cite[Lemma 5.15]{ALLAIRE-2007}.
  
Let us fixe 
$\chi\in\mathcal{U}_{ad}^*(\beta)$. Then let us take $h\in L^{\infty}(\Gamma,\mu)$ such that $\chi+h\in\mathcal{U}_{ad}^*(\beta)$ (for such $h$ a necessarily condition  is $\int_\Gamma h d\mu=0$). 
Then for all $t\in [0,1]$  
\begin{equation}\label{EqCondchihat}
\hat{\chi}(t):= \chi+t \,h \in \mathcal{U}_{ad}^*(\beta)\quad \hbox{ and satisfies }	\quad \beta\le \| \hat{\chi}(t)\|_{L^\infty(\Gamma,\mu)}. 
\end{equation}

By Theorem~\ref{ThWPchi} problem~\eqref{frechet_derivability} is weakly well-posed: there exists a unique solution
$\psi\in V(\Omega)$ such that it holds~\eqref{EqVFWder}.
Then we consider $\hat{u}(t)=u(\hat{\chi}(t))$, $i.e.$ the weak solution of~\eqref{EqVFchi} found for $\chi:=\hat{\chi}(t)$.
In addition, there hold $\hat{u}(0)=u(\chi)$ and $\hat{u}(1)=u(\chi+h)$.
Thus, we find the variational formulation for $\Psi:=u(\chi+h)-u(\chi)-\psi$:
\begin{multline}\label{EqVFdifetw}
	\forall \phi \in V(\Omega) \quad \int_\Omega \nabla \Psi\nabla \overline{\phi} \dx -k^2 \int_\Omega \Psi \overline{\phi} \dx+ \int_\Gamma \alpha \chi \mathrm{Tr} \Psi \mathrm{Tr}\overline{\phi}d \mu\\  =-\int_\Gamma  \alpha h \left(\mathrm{Tr}u(\chi+h)- \mathrm{Tr}u(\chi)\right)\mathrm{Tr}\overline{\phi}d \mu,
\end{multline}
which by Theorem~\ref{ThWPchi} is also well-posed in $V(\Omega)$ and it holds~\eqref{EqApriorichi} with $F=0$ and\\ $\eta=-\alpha h \left(\mathrm{Tr}u(\chi+h)- \mathrm{Tr}u(\chi)\right)$. 
Hence, using~\eqref{EqApriorichi} with a constant $\hat{C}>0$ independent on $\chi$ thanks to Lemma~\ref{LemIndepChi}, we have the estimate
\begin{multline*}
\|u(\chi+h)-u(\chi)-\psi\|_{V(\Omega),\chi}
\le \hat{C}\|\alpha h \left(\mathrm{Tr}u(\chi+h)- \mathrm{Tr}u(\chi)\right)\|_{L^2(\Gamma,\mu)}
\le\\C(\alpha,\|\operatorname{Tr}\|_{\mathcal{L}(V(\Omega),L^2(\Gamma,\mu))}) \| h\|_{L^\infty(\Gamma,\mu)} \| u(\chi+h)- u(\chi)\|_{V(\Omega),\chi}. 
\end{multline*}
Thanks to~\eqref{EqEstDiffu} and then once again by Lemma~\ref{LemIndepChi}, for fixed $k$, $\alpha$, $\eta$ and $F$ (which norms contribute to the general constant), we have
\begin{multline}\label{EqContinU}
	\|u(\chi+h)-u(\chi)\|_{V(\Omega),\chi}\le C(\alpha)\|h\|_{L^\infty(\Gamma,\mu)}\|\mathrm{Tr}u(\chi)\|_{L^2(\Gamma,\mu)}\\
	\le C(\alpha,\|\operatorname{Tr}\|)\|h\|_{L^\infty(\Gamma,\mu)}\|u\|_{V(\Omega),\chi}\le C(\alpha,\|\operatorname{Tr}\|, \|\eta\|_{L^2(\Gamma,\mu)}, \|F\|_{L^2(\Omega)})\|h\|_{L^\infty(\Gamma,\mu)}, 
\end{multline}
where we also used  the continuity of the trace operator. 
Let us notice that estimate~\eqref{EqContinU} implies the continuity of $u$ on $\chi$ in the strong topology of $L^\infty(\Gamma,\mu)$, which naturally follows from the weak$^*$ continuity obtained in the proof of Lemma~\ref{LemContWeak}.

Consequently, there exists a constant $C>0$ such that $\|u(\chi+h)-u(\chi)-\psi\|_{V(\Omega)}\le C   \| h\|_{L^\infty(\Gamma,\mu)}^2,$
which means that 
$$u(\chi+h)=u(\chi)+\psi+o(h) \hbox{ with } \lim_{\|h\|_{L^\infty(\Gamma,\mu)}\to 0} \frac{\|o(h)\|_{V(\Omega)}}{\|h\|_{L^\infty(\Gamma,\mu)}}=0.$$
This is exactly the differentiability by Fréchet since the application $h\mapsto \psi$ (see~\eqref{EqVFWder} and~\eqref{EqOpT}) is linear continuous from $L^\infty(\Gamma,\mu) \to V(\Omega)$.
\end{proof}

        By the Lagrangian method, 
        the weak solution of the corresponding variational formulations,
        $\langle \frac{\del \mathcal{L}}{\del w_R}(\chi, u_R, u_I,q_R,q_I),\phi_R\rangle=0$  for all $\phi_R\in V(\Omega)$  and  $\langle \frac{\del \mathcal{L}}{\del w_I}(\chi, u_R, u_I,q_R,q_I),\phi_I\rangle$ $=0$  for all $\phi_I\in V(\Omega)$,
        is denoted by $p=p_R+i p_I\in V(\Omega)$ (see~\eqref{EqVFAdjProb}) and is called the solution of the adjoint problem.  
 
        Finally, we calculate by~\eqref{EqDerivLagrChi} the  derivative of $J$ over $\chi$, evaluated in the direction $h\in L^\infty(\Gamma,\mu)$, according to the assumptions of Lemma~\ref{LemDiffer},  and obtain~\eqref{EqJder}. 
This finishes the proof of Theorem~\ref{ThMainROptParam}.

\section{Numerical optimization}\label{SecNumR} 
We consider the approximation of problem~\eqref{EqHelmChi} by the finite volumes (or the cell-centered finite difference method with unknowns in the center of the mesh cells with the second order convergence rate) on a square domain $\Omega=]0,1[^2$ represented in  Figure~\ref{fig:geom} along with the chosen boundary conditions. We could model wave propagation in a tunnel or a room with the reflective ground and cell and a partially absorbent wall opposite the noise source. 
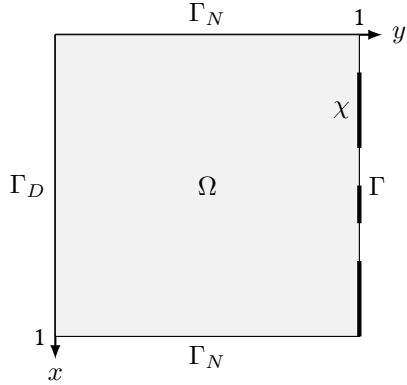
\begin{figure}[!htb]
        \centering
        \begin{tikzpicture}
            \draw [->,thick, >=latex] (0,4) -- (0,-0.3) node [below] {\(x\)};
            \draw [->,thick, >=latex] (0,4) -- (4.3,4) node [right] {\(y\)};
            \draw (0,0) -- (4,0);
            \draw (4,4) -- (4,0);
            \node [above] at (4,4) {1};
            \node [left] at (0,0) {1};
            \filldraw [fill=gray!10] (0,0) coordinate (0,0) -- (0,4) -- (4,4) -- (4,0) -- (0,0);
            \node [left] at (4,3) {\(\chi\)};
            \draw [ultra thick] (4,0) -- (4,1);
            \draw [ultra thick] (4,1.5) -- (4,2);
            \draw [ultra thick] (4,2.5) -- (4,3.5);
            \node [right] at (4,2) {\(\Gamma\)};
            \node [left] at (0,2) {\(\Gamma_D\)};
            \node [above] at (2,4) {\(\Gamma_{N}\)};
            \node [below] at (2,0) {\(\Gamma_{N}\)};
            \node at (2,2) {\(\Omega\)};
        \end{tikzpicture}        
        \caption{Geometry of the domain $\Omega$ for the numerical experiments.}\label{fig:geom}
    \end{figure}
    We set for the wave speed in the air $c=\SI{340}{m.s^{-1}}$. We perturb the system by the non-homogeneous Dirichlet boundary conditions $u|_{\Gamma_{Dir}}=g$, where $g(x)=e^\frac{(x-0.5)^2}{\sqrt{2\pi}0.5^2}$ is a centered Gaussian, and take in~\eqref{EqHelmChi} $F=\eta=0$. For the volume fraction of the absorbing material on $\Gamma$, we chose $\beta=0.5$.
    We are searching to minimize $\hat{J}$ (see~\eqref{EqTotalJ}) on the audible frequencies $I=[20Hz,1000Hz]$ for the case $J(f,\chi)=\int_\Omega |u(f,\chi)|^2 dx$, $i.e.$ taking $A=1$ and $B=C=0$ in ~\eqref{EqJ}. Then, the associated adjoint system becomes
    \begin{equation}\label{EqHelmAdj}
  \left\{
    \begin{array}{ll}
        (\Delta+k^2) p=-2\overline{u}  \mbox{ on }  \Omega; \\
        \frac{\partial p}{\partial n} = 0  \mbox{ on }    \Gamma_{Neu};\quad
         p = 0 \mbox{ on } \Gamma_{Dir};\quad
        \frac{\partial p}{\partial n}+\alpha(f)\chi p = 0\mbox{ on } \Gamma.
    \end{array}
\right.  
\end{equation}
    We use the same coefficient $\alpha(f)$, as it was found in~\cite[Appendix B]{MAGOULES-2021} for the  ISOREL porous material (see~\cite[Fig.~2]{MAGOULES-2021}).   
    In this section, we denote by $\lambda$ the wavelength of the wave. We penalize the minimum length of connected parts of $\Gamma$ with $\chi=0$, denoted by $\mathrm{minlen}(\chi)$ to be not less than $0.1$ for $\chi\in [0,1]$. This helps to avoid discretization problems in solving the Helmholtz system (the direct and the adjoint) and to take correctly into account the changes of Robin absorbing boundary condition on the homogeneous Neumann one.   For the numerical experiments, we define the spatial step of the mesh discretization $h=\Delta x=\Delta y$ equal to $\frac{\min(\lambda,\mathrm{minlen}(\chi))}{80}$ as a function of $\lambda$ and $\mathrm{minlen}(\chi)$. By our numerical tests for a fixed $\chi$, this choice of $h$ corresponds to $0.6\%$ relative $L^2$-error between an exact and the calculated solution for a fixed frequency and to $0.2\%$ for the integrated on $I$ corresponding energies, $\int_I J(f,\chi) d f$. 

    On Figure~\ref{fig:chi_50}, we present the solutions of the direct Helmholtz problem at the fixed frequency $f=\SI{200}{Hz}$ with different distributions $\chi$ on $\Gamma$. We see that the case of the full absorbent $\Gamma$ has much more red colors corresponding to a higher energy than for two other cases of $50\%$ less porous material.  
\begin{figure}[!htb]
    \centering
    \includegraphics[width=0.25\textwidth]{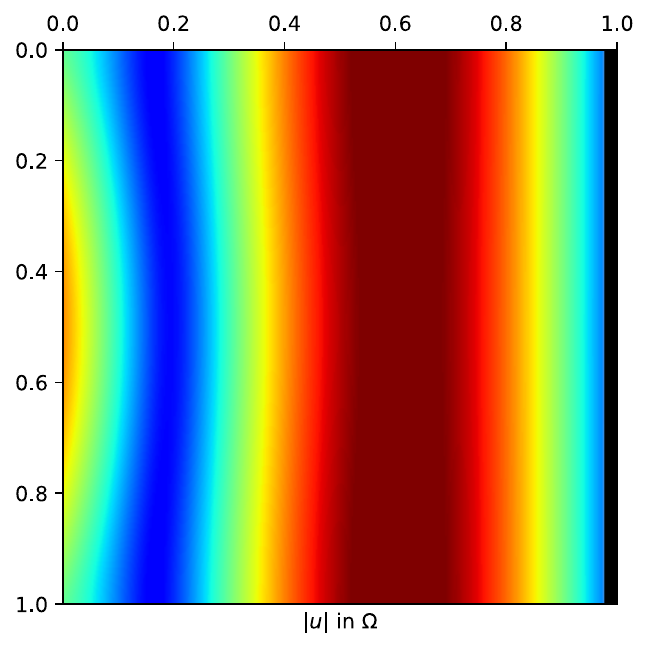}
    \includegraphics[width=0.25\textwidth]{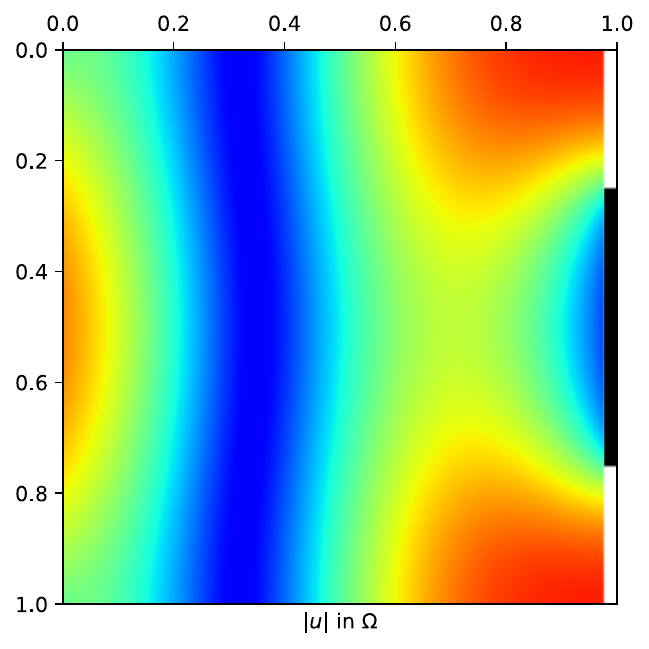}
    \includegraphics[width=0.25\textwidth]{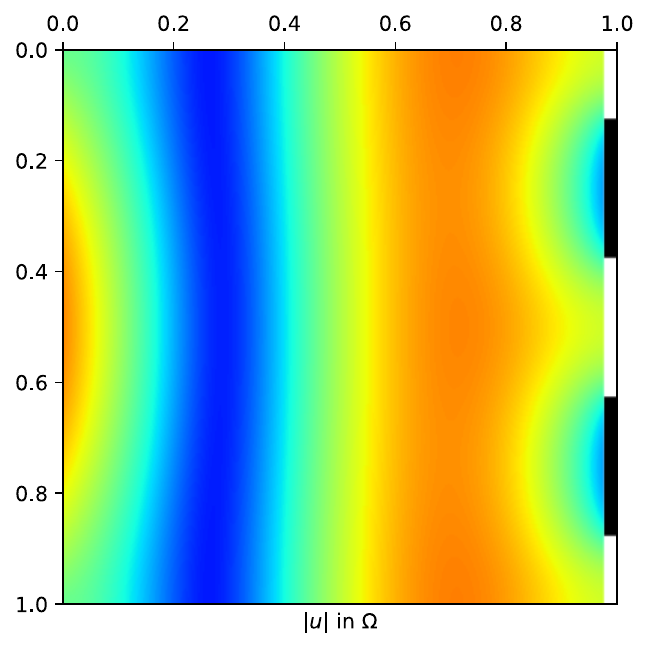}
    \caption{Examples of $|u_j|$ ($j=0,1,2$), the solutions of the Helmholtz direct problem for different distributions of $\chi_j$,  at the fixed frequency $\SI{200}{Hz}$. The color scaling is the same for three cases: strong blue is $0$, and strong red is the biggest value. On $\Gamma$, black is the absorbent material, and white is the reflective one. }\label{fig:chi_50}
    \hfill
\end{figure}

\subsection{Optimization algorithm}\label{ss:algo_gradient}
To calculate the optimal absorbent distribution $\chi$, we apply the Barzilai-Borwein gradient method with modified step size, inspired by Newton's method~\cite{BARZILAI-1988,DAI-2002}, to the relaxed problem on $U_{ad}^*(\beta)$ with $\chi\in [0,1]$ for $J^*$ or $\hat{J}^*$. For the case of the frequency range optimization of $\hat{J}^*$, we take the discretization of $I$ with $400$ uniformly spaced frequencies $f_i$, $i=1,\ldots,400$. 
The gradient descent algorithm computes a series of distributions $\chi_1, \chi_2, \ldots$ and stops when they converge. It goes from $\chi_j$ to $\chi_{j+1}$ with the following steps (formulated here for $\hat{J}^*$):
\begin{itemize}
    \item Simulation:
    \begin{itemize}
        \item for each chosen frequency $f_i$, solve the direct  and the adjoint Helmholtz problems~\eqref{EqHelmChi}, \eqref{EqHelmAdj} with parameters $h$, $\alpha(f_i)$, $k=\frac{2\pi f_i}{c}$, $\chi_j$ and $g$;
        \item to find $J^*(f_i,\chi_j)$ for all $i$ and approximate $\hat{J}^*(\chi_j)$ using the composite trapezoidal rule;
        \item compute $D_\chi{\hat{J}^*}(\chi_j)$ by~\eqref{EqDerJhat}. 
    \end{itemize}
    \item Gradient descent method: apply the Barzilai-Borwein gradient method 
    with\\ $(\gamma_j, D_\chi{\hat{J}^*}(\chi_j))$ (by $\gamma_j$ is denoted a step size of the gradient descent algorithm) to obtain the corrected gradient $\tilde{D}_\chi{\hat{J}^*}(\chi_j)$ and the new step size $\gamma_{j+1}$.
    \item Stop test: if $\|\gamma_j\tilde{D}_\chi{\hat{J}^*}(\chi_j)\| \leq \delta_1$ or $\frac{1}{m+1}\sum_{i=j-m}^{j}|\frac{\hat{J}^*(\chi_i)-\hat{J}^*(\chi_{i-1})}{\hat{J}^*(\chi_{i-1})}| \leq \delta_2$ for some fixed small strictly positive constants $\delta_1$, $\delta_2$ and a natural $m$, end the algorithm.
    \item New distribution: $\chi_{j+1} = \mathcal{P}(\chi_j - \gamma_j\tilde{D}_\chi{\hat{J}^*}(\chi_j))$, where $\mathcal{P}: (\Gamma \to \mathbb{R}) \longrightarrow{} (\Gamma \to [0,1])$ is the projection on $U^*_{ad}(\beta)$ with the condition 
    \begin{equation}\label{EqCondb}
    	\frac{1}{\mu(\Gamma)}\int_\Gamma\mathcal{P}(\chi) d\mu =\beta, \quad \beta \in ]0,1[.
    \end{equation}
\end{itemize}
For the projection on $U^*_{ad}(0.5)$ we use the following sigmoid projection 
$\mathcal{P}$: $\chi\mapsto \Phi\circ (\chi+\ell_\chi)$ with $\ell_\chi$ ensuring~\eqref{EqCondb} with $\beta=0.5$ ($\mu(\Gamma)=1$), and 
$\Phi(x)=\frac{1}{1+\exp(-8 (x-0.5))}$. 

Once the algorithm is stopped, let $\chi_n$ be the absorbent distribution at the end of the loop. Thus,  $\chi_n\in U^*_{ad}(0.5)$ is the optimal distribution, denoted by $\chi^{opt}$, taking the values in $[0,1]$, which realize a local (ideally global) minimum of the energy functional $\hat{J}^*$.

Finally, we project it on $U_{ad}(0.5)$ to obtain a characteristic function: 
\[\chi^{opt}_{projected}=\mathcal{P}_f(\chi^{opt})\in U_{ad}(0.5).\] 
We define $\mathcal{P}_f$ in the same way as previously $\mathcal{P}$, taking this time instead of $\Phi$ the function $$\Phi_f(x)=\left\{\begin{array}{cl}
    0 & \text{if } x \leq 0.5, \\
    1 & \text{if } x > 0.5.
    \end{array}
    \right.$$ This projection is equivalent to sorting all the values of $\chi$ and setting the largest ones to $1$ and the rest to $0$ so as to reach the desired volume $\int\limits_\Gamma\chi \,d\mu$.

\begin{remark}
	The choice of the sigmoid projection $\mathcal{P}$ makes the final projection of the optimal $\chi^{opt}\in U_{ad}^*(0.5)$ on $U_{ad}(0.5)$, with the only two possible values $0$ and $1$, less brutal as, for instance, for the rectified linear projection (see Figure~\ref{FigProjDif}), and thus somewhere less destroys the found optimal absorbing performances of $\chi^{opt}$.
\end{remark}
 \begin{figure}[!htb]
    \centering
     \begin{tikzpicture}
    \begin{axis}[
        width = 5cm,
        height = 3.5cm,
        domain     = -0.5:1.5,
        ymin       = 0,
        ymax       = 1.2,
        samples    = 70,
        axis lines = middle,
        grid,
        xtick = {0,.5,1},
        xticklabels = {$0$, $\frac 12$, $1$},
        ytick = {0,.5,1},
        yticklabels = {$0$, $\frac 12$, $1$},
        yticklabel style={above, left}
        ]
    \addplot [mark = none, color = red, line width=1.5]
        {min(max(x,0),1)};
    \end{axis}
    \end{tikzpicture} $ \; \; $
    \begin{tikzpicture}
    \begin{axis}[
        width = 5cm,
        height = 3.5cm,
        domain     = -0.5:1.5,
        ymin       = 0,
        ymax       = 1.2,
        samples    = 70,
        axis lines = middle,
        grid,
        xtick = {0,.5,1},
        xticklabels = {$0$, $\frac 12$, $1$},
        ytick = {0,.5,1},
        yticklabels = {$0$, $\frac 12$, $1$},
        yticklabel style={above, left}
        ]
    \addplot [mark = none, color = red, line width=1.5]
        {1/(1+exp(-8*(x-0.5))};
    \end{axis}
    \end{tikzpicture}
    \caption{Rectified linear projection function on the right and sigmoid projection $\Phi$ on the left.}\label{FigProjDif}
    \end{figure}
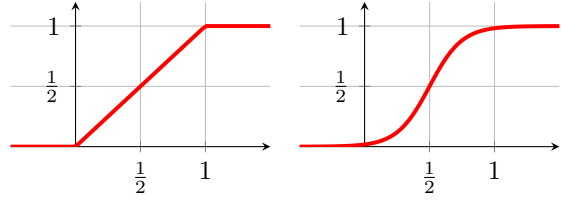

 To start the optimization algorithm, we fix the initial distribution $\chi_0$, corresponding to the defined $\beta=0.5$, presented as a function of the natural parameter of $\Gamma$ on Fig.~\ref{fig:chi_0} (on the left). The corresponding solution (its real and imaginary parts) for $f=\SI{500}{Hz}$ is given on the right-hand part of Fig.~\ref{fig:chi_0}.
 \begin{figure}[!htb]
     \centering
     \includegraphics[width=0.45\textwidth]{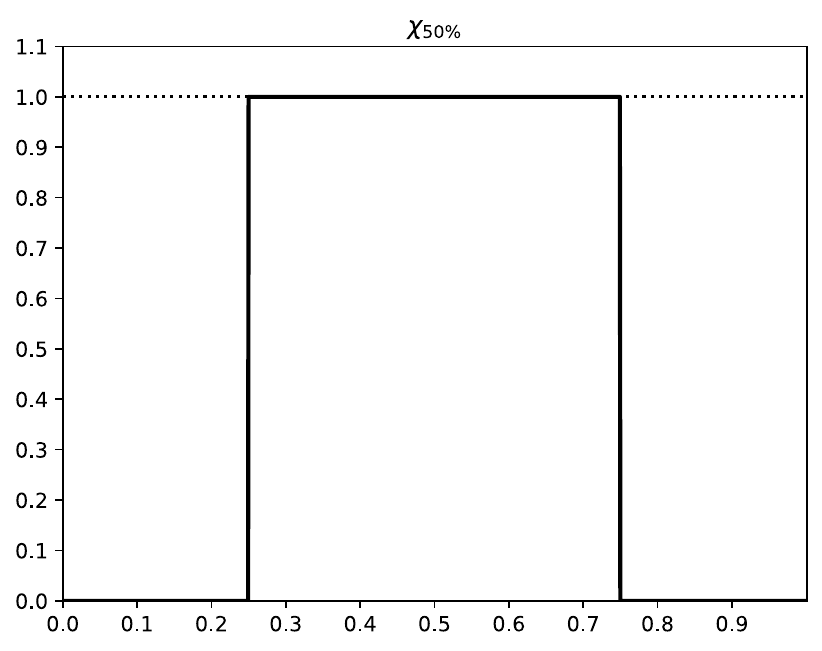}
     \includegraphics[width=0.5\textwidth]{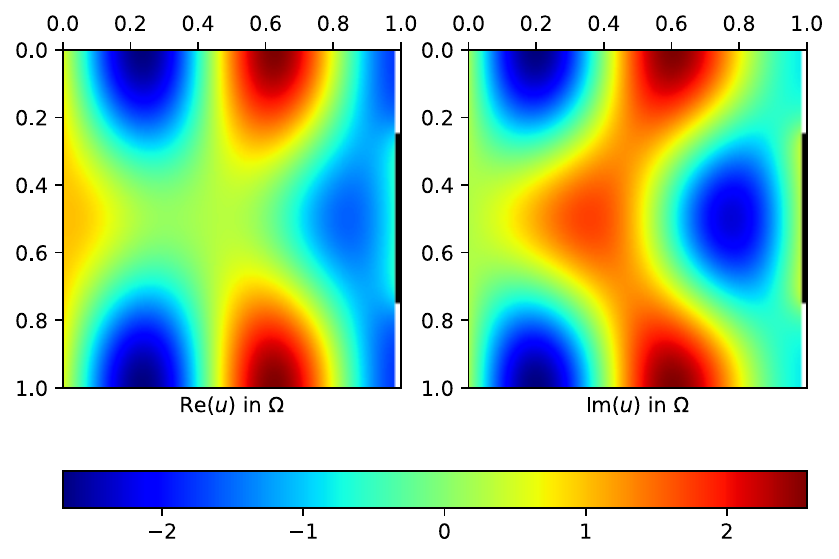}
     \caption{The initial absorbent distribution $\chi_0(y)$, $y\in [0,1]$, on $\Gamma$ on the left, and the corresponding real and imaginary parts of the solution of the Helmholtz problem found for the frequency $f=\SI{500}{Hz}$ on the right. A black line gives the presence of the porous medium on $\Gamma$.}
     \label{fig:chi_0}
 \end{figure}
 Therefore, to have a reference point for the energy optimization, we plot on Figure~\ref{fig:initial_energy_spectrum} the corresponding energy $J(\chi_0,f)$ for the initial distribution (blue line) 
 and compare it to the energy of the fully absorbent $\Gamma$ ($\chi(x)=1$ for all $x\in \Gamma$), presented by the black line.
 The main goal is to find a distribution $\chi\in U_{ad}(0.5)$ such that the energy $J(\chi)$ on the interval of chosen frequencies would be not worse or even smaller than the energy corresponding to the fully absorbent $\Gamma$.
 \begin{figure}[!htb]
   \centering
   \includegraphics[width=0.7\textwidth]{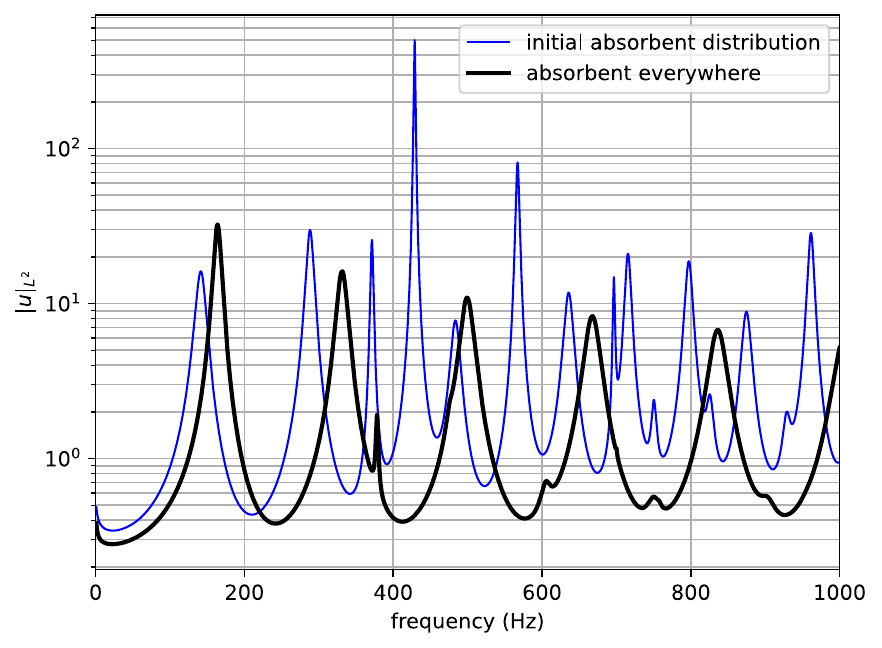}
    \caption{Energies on the frequency range for the initial absorbent distribution $\chi_0$ and for the fully absorbent $\Gamma$.}
     \label{fig:initial_energy_spectrum}
 \end{figure}
First, we consider the optimization for only one fixed frequency and study the dependence of the corresponding energy on $I$. Secondly, we consider a frequency discretization of the interval $I$ and study the minimization of $\hat{J}$ and $\hat{J}^*$.
 \subsection{Optimization on a single frequency}\label{SubSecOptNumSF}
We chose three typical frequencies: $f=\SI{100}{Hz}$ in the low frequencies, $f=\SI{500}{Hz}$ in the mean frequencies and $f=\SI{1000}{Hz}$ as a high frequecy.
Each time, the optimization procedure is started with $\chi_0$ and $\beta=0.5$, as explained previously.
Numerical results are presented on Figure~\ref{FigCompOneF}. We notice that for the middle $\SI{500}{Hz}$ and high $\SI{1000}{Hz}$ frequencies, the optimal shapes $\chi^{opt}$ from $U_{ad}^*(0.5)$ have a very similar form to a characteristic function, but not for the low frequency $\SI{100}{Hz}$. Therefore, by the continuity property of the energy on $\chi$, we do not see significant changes in the values of the energy on $I$ for $\chi^{opt}$ and $\chi^{opt}_{projected}$ obtained for  $f=\SI{500}{Hz}$ or $\SI{1000}{Hz}$. However, we see them in the case of $f=\SI{100}{Hz}$.
 \begin{figure}[!htb]
\centering
    \includegraphics[width=0.325\textwidth]{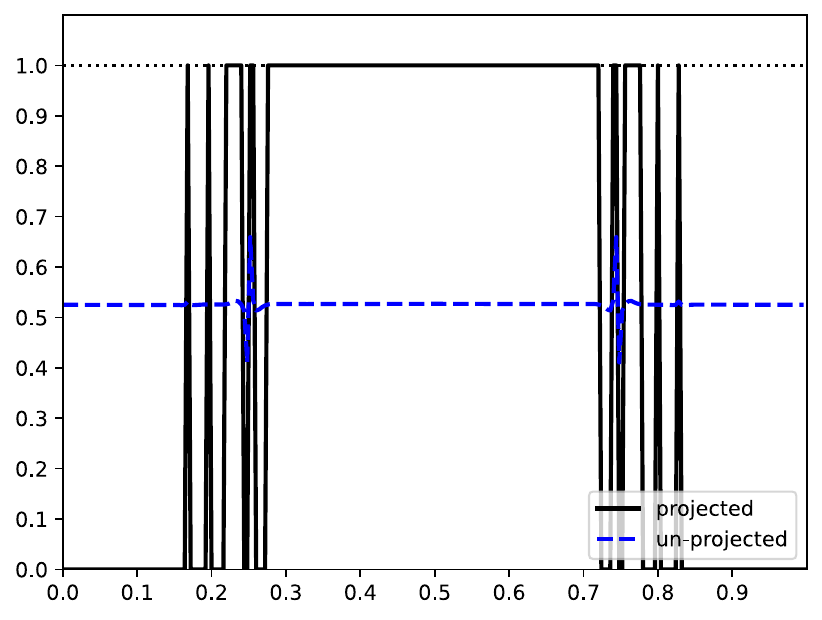}
         \includegraphics[width=0.325\textwidth]{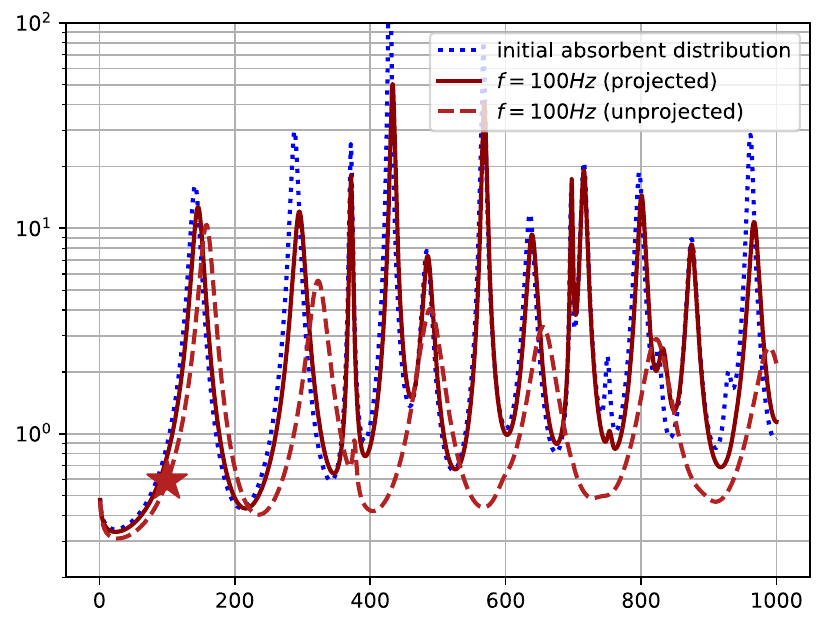}
     \includegraphics[width=0.325\textwidth]{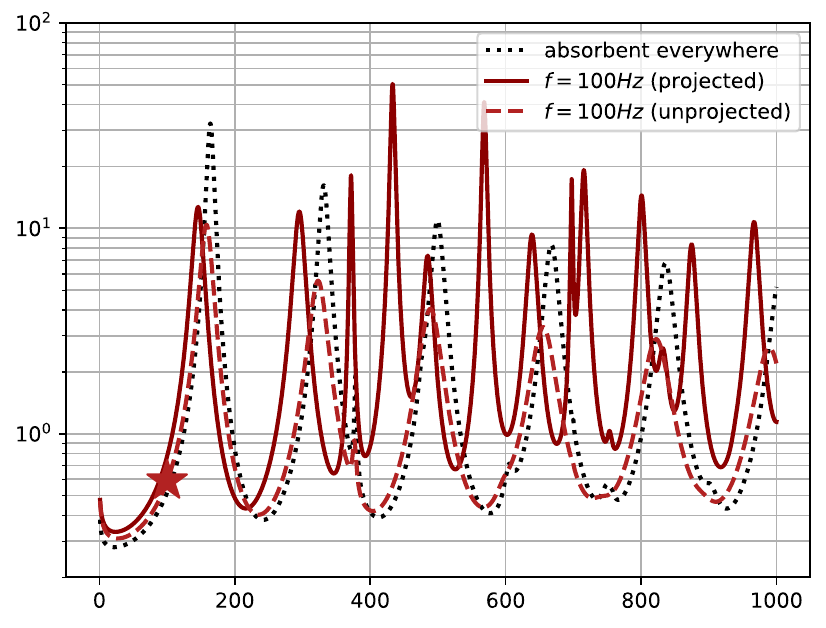}\\
     \includegraphics[width=0.325\textwidth]{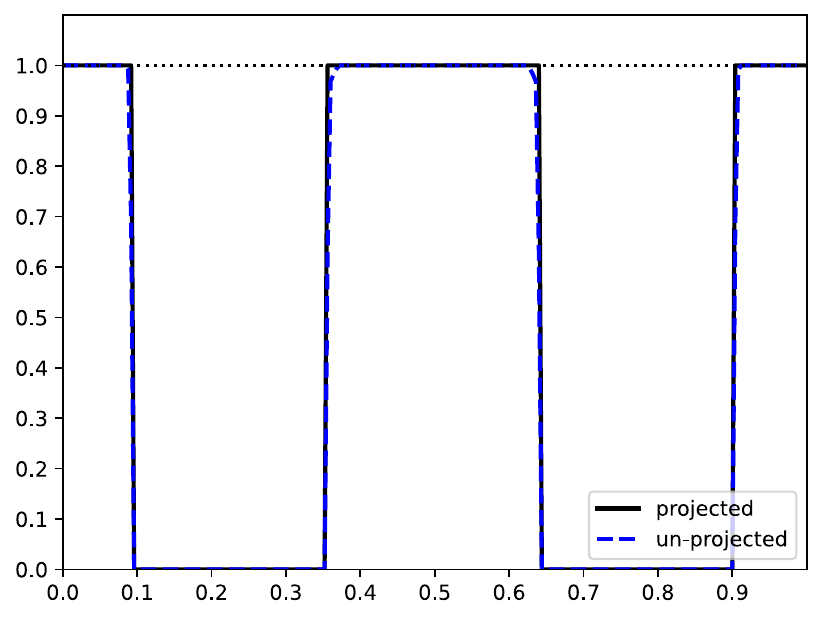}
     \includegraphics[width=0.325\textwidth]{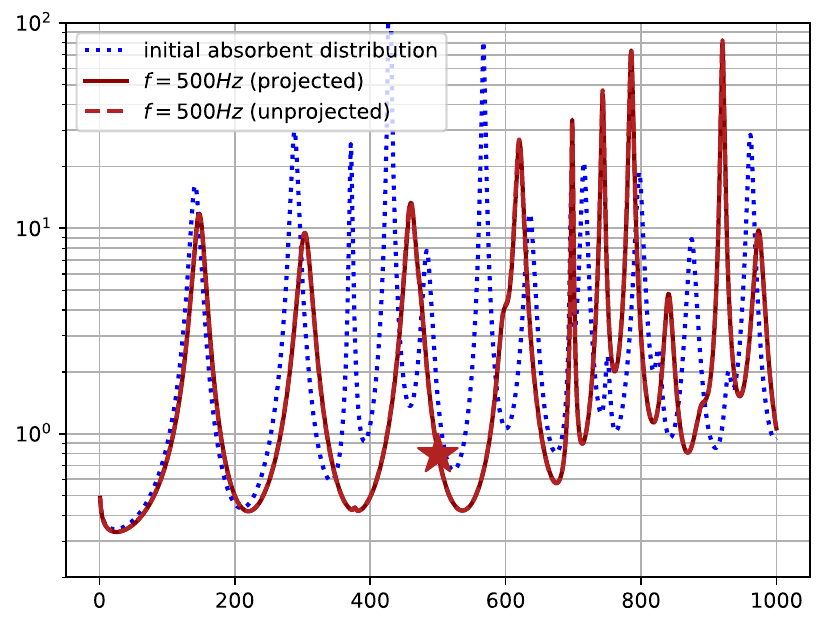}
     \includegraphics[width=0.325\textwidth]{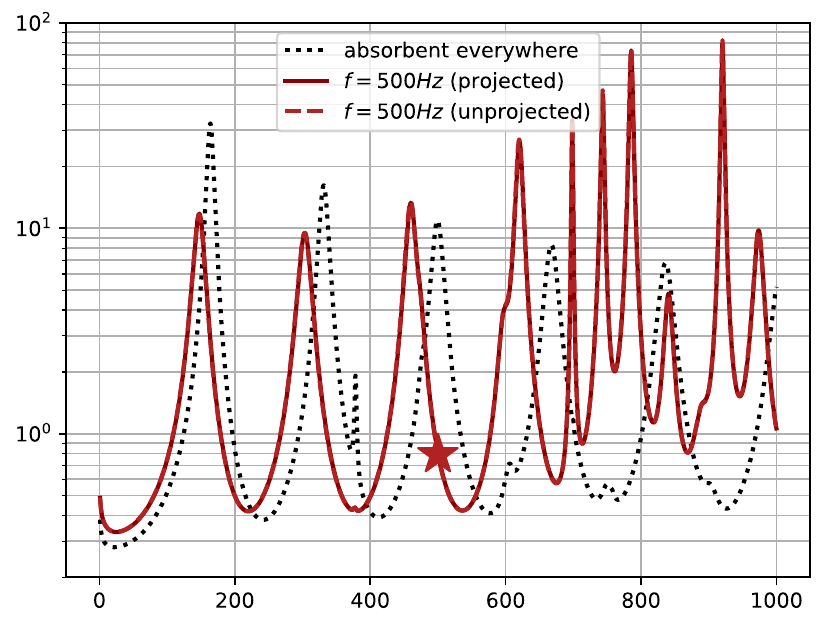}\\
     \includegraphics[width=0.325\textwidth]{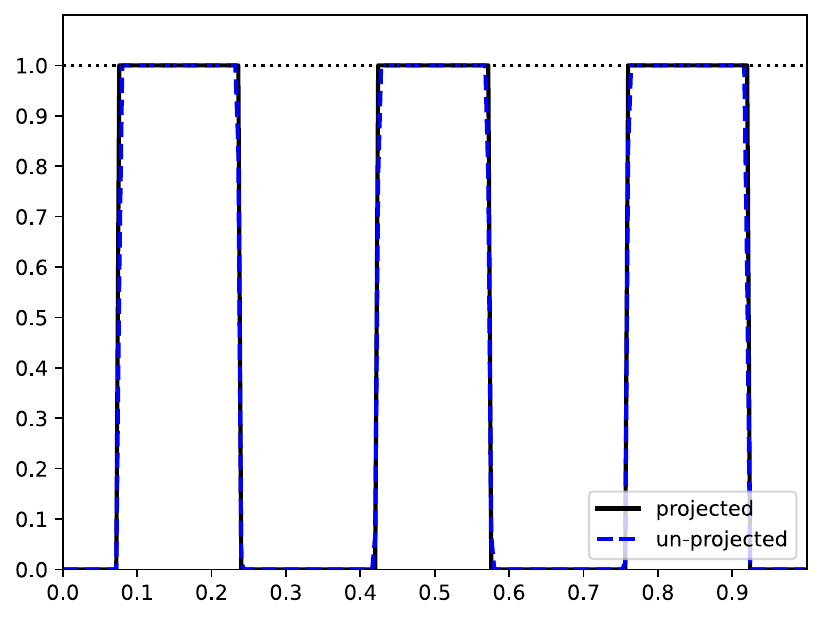}
     \includegraphics[width=0.325\textwidth]{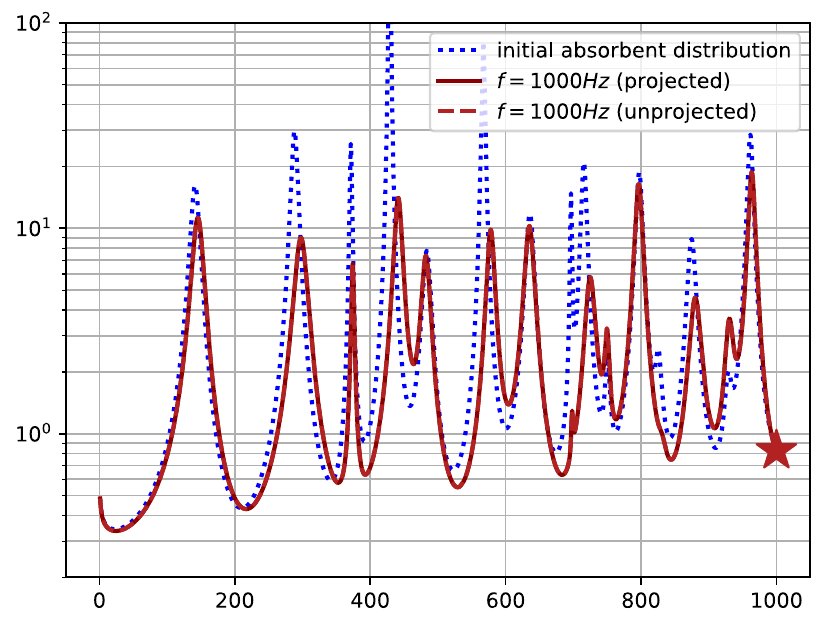}
     \includegraphics[width=0.325\textwidth]{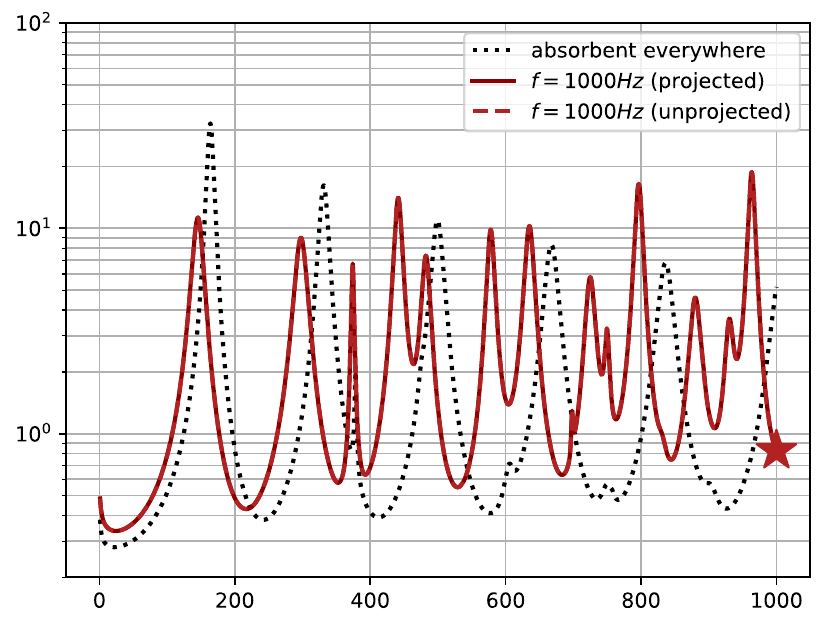}
    \caption{Optimization results for three frequencies (from the top line to the bottom line) for   $\SI{100}{Hz}$, $\SI{500}{Hz}$ and $\SI{1000}{Hz}$ respectively. From the left to the right: $\chi^{opt}\in U_{ad}^*(0.5)$ (blue dotted line) and $\chi^{opt}_{projected}\in  U_{ad}(0.5)$ (black solid line) as the functions of the natural parametrization of $\Gamma$ on the left; comparison between the energy for the obtained $\chi^{opt}$ and $\chi^{opt}_{projected}$ with the initial $\chi_0$ on the middle, and with the fully absorbing case $\chi_{100\%}$ on the right. A star indicates the frequency of the optimization.}\label{FigCompOneF}
    \end{figure}
    In the mean, by Figure~\ref{FigCompOneF}, the energy found for the optimized $\chi$ for the frequency $f=\SI{1000}{Hz}$ with $50\%$ of porous material is not too more significant than the energy of the fully absorbing case.

 However, without a surprise, optimizing one frequency results in the displacement of the peaks without significantly reducing the total integrated energy. Thus, there is a need to optimize on several frequencies at once.
 \subsection{Optimization on a frequency range}\label{SubSecOptNumRF}
 Keeping without changes previously defined numerical parameters, this time we take $\SI{400}{}$ uniformly spaced frequencies on $I$. %
 Then, we optimize the acoustical energy integrated over $I$ (the composite trapezoidal rule approximates the integral, see Subsection~\ref{ss:algo_gradient}). 

To see if the implemented gradient descent method gives satisfying performances, we also implemented a  genetic algorithm, based on the covariance matrix adaptation evolution strategy (CMA-ES)~\cite{AHAMED-2013}, and compared the optimality properties of two obtained distributions of $\chi$   given on Figure~\ref{FigGlobalOpt}. 
\begin{figure}[!htb]
    \centering
       \includegraphics[width=0.49\textwidth]{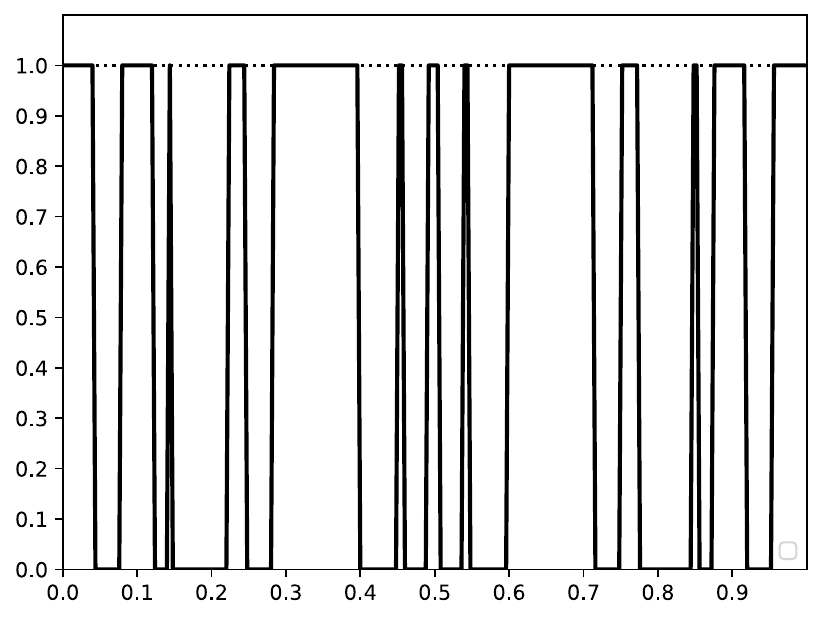}
        \includegraphics[width=0.49\textwidth]{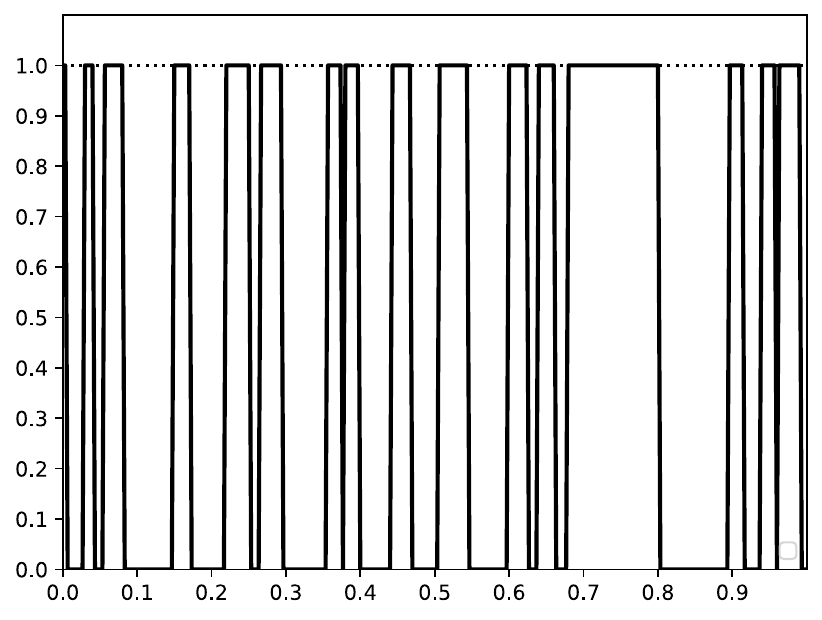}
    \caption{Optimal absorbent distribution $\chi^{opt}_{projected}$ as the function of the natural parametrization of $\Gamma$, found by gradient descent method on the left and genetic algorithm on the right for the frequency range optimization with $50\%$ of porous material.}\label{FigGlobalOpt}
    \end{figure}
    \begin{figure}[!htb]
\centering
    \includegraphics[width=0.7\textwidth]{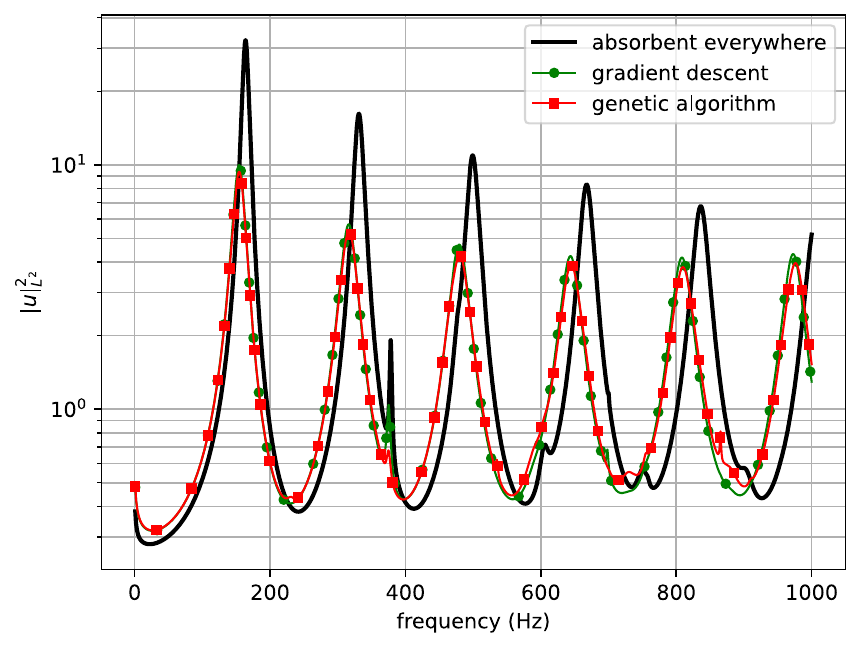}
    \caption{\label{fig:comparaison_gen} 
    Energies on the frequency range $I$ for the optimal absorbent distribution are found using the gradient descent (green line with circles) and using the genetic algorithm (red line with squares) to compare to the energy corresponding to the fully absorbing boundary $\Gamma$ (black line). }
    \end{figure}
   Figure~\ref{fig:comparaison_gen} shows that the energies for the optimal shapes given on Figure~\ref{FigGlobalOpt}, obtained by different methods, are almost the same and both better than the case of $100\%$ absorbing boundary.
   We give now the values of  integrals over $I$ for all energies presented on Fig.~\ref{fig:comparaison_gen}:
   \begin{center}
    \begin{tabular}{lc}
    \toprule
    Distribution & Total energy $\hat{J}$\\
    \midrule
    Absorbent everywhere & 2093 \\
   Gradient descent & 1475 \\
    Genetic algorithm & 1440 \\
    \bottomrule
    \end{tabular}
\end{center}
\bigskip
We found approximately the same total energy $\hat{J}$ for two distributions presented on Figure~\ref{FigGlobalOpt}.
In particular, Figure~\ref{fig:comparaison_gen} shows that with $50\%$ of absorbing material, it is possible to create a distribution of porous material that absorbs more efficiently ($\approx 31\%$ better)  the acoustical energy than with $100\%$ of absorbing material.

%
%
%

\section*{Acknowledgment}
The authors thank Mathieu Boschat for his preliminary work, which was done under the joint supervision of A. Rozanova-Pierrat and F. Magoules, on the subject during his studies at CentraleSupélec. The authors are very grateful to Pascal Omnes for the discussions and fruitful ideas related to finite volume discretization and Eric Savin for his enthusiasm and interest in the solved problem.
The general physical interest in problems with fewer materials and better absorbing performances was initially pointed out by Bernard Sapoval during his collaboration with A. Rozanova-Pierrat.
A. Rozanova-Pierrat thanks Michael Hinz and Alexander Teplayev for collaborating on the non-Lipschitz functional analysis topics.
\bibliographystyle{siamplain}
\label{bib:sec}
\bibliography{/home/anna/Documents/Statii/bibtex/biblio.bib}

\end{document}